\documentclass[12pt,a4paper]{amsart}
\usepackage{amsmath}
\usepackage{amssymb}
\usepackage[english]{babel}
\usepackage{verbatim}
\usepackage[shortlabels]{enumitem}
\usepackage{color}
\usepackage{tikz-cd}

\newtheoremstyle{mio}%
	{}{} 
	{\itshape}{} 
	{\bfseries}{.}{ } 
	{#1 #2\thmnote{~\mdseries(#3)}} 
\theoremstyle{mio}
\newtheorem{teor}{Theorem}[section]
\newtheorem{cor}[teor]{Corollary}
\newtheorem{prop}[teor]{Proposition}
\newtheorem{lemma}[teor]{Lemma}
\newtheorem{defin}[teor]{Definition}

\newtheoremstyle{definition2}%
	{}{} 
	{}{} 
	{\bfseries}{.}{ } 
	{#1 #2\thmnote{\mdseries~ #3}} 
\theoremstyle{definition2}
\newtheorem{ex}[teor]{Example}
\newtheorem{oss}[teor]{Remark}

\newcommand{\insN}{\mathbb{N}}
\newcommand{\inN}{\in\insN}
\newcommand{\insR}{\mathbb{R}}
\newcommand{\inR}{\in\mathbb{R}}

\newcommand{\Max}{\mathrm{Max}}
\newcommand{\insid}{\mathcal{I}}
\newcommand{\insfracid}{\mathcal{F}}
\newcommand{\insstar}{\mathrm{Star}}
\newcommand{\insfstar}{\mathrm{FStar}}
\newcommand{\inssmstar}{\mathrm{(S)Star}}
\newcommand{\inssemistar}{\mathrm{SStar}}
\newcommand{\inssubmod}{\mathbf{F}}
\newcommand{\insmult}{\mathrm{Mult}}
\newcommand{\insclos}{\mathrm{Clos}}

\newcommand{\val}{\mathbf{v}}
\newcommand{\princ}{\ast}
\newcommand{\ordine}{\preceq}
\newcommand{\ordinestretto}{\prec}
\newcommand{\Jac}{\mathrm{Jac}}
\newcommand{\valut}{\mathbf{v}}
\newcommand{\nz}{\bullet}

\title{Multiplicative closure operations on ring extensions}
\author{Dario Spirito}
\date{\today}
\address{Dipartimento di Matematica e Fisica, Universit\`a degli Studi ``Roma Tre'', Roma, Italy}
\email{spirito@mat.uniroma3.it}

\begin{document}

\begin{abstract}
Let $A\subseteq B$ be a ring extension and $\mathcal{G}$ be a set of $A$-submodules of $B$. We introduce a class of closure operations on $\mathcal{G}$ (which we call \emph{multiplicative operations on $(A,B,\mathcal{G})$}) that generalizes the classes of star, semistar and semiprime operations. We study how the set $\insmult(A,B,\mathcal{G})$ of these closure operations vary when $A$, $B$ or $\mathcal{G}$ vary, and how $\insmult(A,B,\mathcal{G})$ behave under ring homomorphisms. As an application, we show how to reduce the study of star operations on analytically unramified one-dimensional Noetherian domains to the study of closures on finite extensions of Artinian rings.
\end{abstract}

\maketitle

\section{Introduction}
Let $(\mathcal{P},\leq)$ a partially ordered set. A \emph{closure operation} on $(\mathcal{P},\leq)$ is a map $c:\mathcal{P}\longrightarrow\mathcal{P}$ that is extensive ($x\leq c(x)$ for all $x$), order-preserving (if $x\leq y$, then $c(x)\leq c(y)$) and idempotent ($c(c(x))=c(x)$ for all $x$). In commutative algebra, there are several classes of closure operations in use, where the set $\mathcal{P}$ is usually a set of ideals of a ring $R$ or a set of submodules of a given $R$-module; to be connected with the algebraic properties of the ring (or the module), the definition of such classes usually includes a multiplicative property relating the product $x\cdot c(I)$ with the closure $c(xI)$, where $x$ varies among the elements of the ring and $I$ among the set of ideals or submodules considered. The three main classes of closure operations are semiprime operations, star operations and semistar operations (see Section \ref{sect:background} for the definitions).

These three classes are defined very similarly, and it is natural to think that they are somewhat interrelated. Indeed, star and semistar operations are often introduced and studied together (and the latter were actually born as a generalization of the former \cite{okabe-matsuda}); on the other hand, restricting a semistar operation we get a semiprime operation, and this correspondence can partly be inverted \cite{epstein-corresp}. However, the study of these classes is usually pursued in different contexts: star and semistar operations are defined only in the integral domain setting (although they can be generalized: see \cite{elliott-libro}), and their study is often connected with Pr\"ufer domain and their generalizations, while semiprime operations can be defined for arbitrary rings, and are studied especially in the Noetherian context. In particular, a major point of difference is that the most useful semiprime operations have some functorial properties, while star and semistar operations usually behave very badly under quotients, with only some sparse exception \cite{fontana-park}.

In this paper, we define and study \emph{multiplicative operations}, a class of closure operations that encompasses these three classes of closures, allowing to study them in a unified way. Following the concept of \emph{star operation on an extension} introduced by Knebush and Kaiser \cite[Chapter 3, Section 3]{knebush-II}, our setting will always be a ring extension $A\subseteq B$, with no hypothesis on $A$ and $B$, and our partially ordered set will be a subset $\mathcal{G}$ of the set $\inssubmod_A(B)$ of $A$-submodules of $B$ which is, in principle, arbitrary (but we will need some hypothesis to get good properties). The main advantage of our definition is the possibility of varying $B$ and $\mathcal{G}$: for example, with $B=A$ and $\mathcal{G}=\inssubmod_A(A)$ (i.e., $\mathcal{G}$ is the set of all ideals) multiplicative operations reduce to semiprime operations, while if $A$ is a domain, $B$ its quotient field and $\mathcal{G}=\inssubmod_A(B)$ we obtain the semistar operations. In Sections \ref{sect:Gprinc} and \ref{sect:functoriality}, we show how the choice of $B$ and $\mathcal{G}$ influences the set $\insmult(A,B,\mathcal{G})$ of the multiplicative operations on $(A,B,\mathcal{G})$, and how the flexibility of the definition allows to prove functorial properties. We also show how this point of view allows to generalize the study of semiprime operations of discrete valuation domains done in \cite{vassilev_structure_2009} to arbitrary one-dimensional valuation domains (Example \ref{ex:dimV1}).

In Sections \ref{sect:PID} and \ref{sect:Artin}, we specialize to the case where $B$ is a principal ideal domain: in particular, we show how in this context it is possible to relate the multiplicative operations on $(A,B,\mathcal{G})$ with the multiplicative operations on a quotient $(A/I,B/I,\mathcal{G}')$, where $I$ is a common ideal of $A$ and $B$ and $\mathcal{G}'$ depends on $\mathcal{G}$. When $A$ is a one-dimensional Noetherian domain with finite normalization, we show how the set of star operations on $A$ can be interpreted as the set of multiplicative operations on a proper quotient of $A$; in particular, if $B$ is the normalization of $A$ and we fix the length $\ell_A(B/A)$, we only get finitely many possibilities for the set of star operations, giving an ``high-level'' explanation of some results of \cite{houston_noeth-starfinite}. This line of investigation, together with explicit bounds on the number of star operations, will be pursued in a forthcoming paper \cite{asymptotics-star}.

\section{Notation and background}\label{sect:background}
Throughout the paper, all rings will be commutative and unitary.

Given a ring $A$, we denote by $\insid(A)$ the set of ideals of $A$; if $M$ is an $A$-module, we denote by $\inssubmod_A(M)$ the set of $A$-submodules of $M$. If $D$ is an integral domain, a \emph{fractional ideal} of $D$ is a $D$-submodule $I$ of the quotient field of $D$ such that $dI\subseteq D$ for some $d\in D\setminus\{0\}$; we denote by $\insfracid(D)$ the set of fractional ideals of $D$. If $\mathcal{G}\subseteq\inssubmod_A(M)$, we set $\mathcal{G}^\nz:=\mathcal{G}\setminus\{(0)\}$.

\begin{defin}
Let $(\mathcal{P},\leq)$ a partially ordered set. A \emph{closure operation} on $(\mathcal{P},\leq)$ is a map $c:\mathcal{P}\longrightarrow\mathcal{P}$ such that, for every $x,y\in\mathcal{P}$:
\begin{itemize}
\item $x\leq c(x)$;
\item $x\leq y$ implies $c(x)\leq c(y)$;
\item $c(c(x))=c(x)$.
\end{itemize}
\end{defin}

We denote by $\insclos(\mathcal{P})$ the set of closure operations on $\mathcal{P}$. This set is endowed with a natural partial order, where $c_1\leq c_2$ if and only if $c_1(x)\leq c_2(x)$ for every $x\in\mathcal{P}$, or equivalently if and only if $x=c_2(x)$ implies that $x=c_1(x)$.

When $\mathcal{P}$ is a subset of $\inssubmod_A(M)$ (for some $A$ and $M$) we consider it a partially ordered set under the containment order; furthermore, we write the image of $I\in\inssubmod_A(M)$ under the closure operation $c$ by $I^c$.

The most common classes of closure operations in commutative algebra are semiprime, star and semistar operations.
\begin{itemize}[itemsep=.25em]
\item A \emph{semiprime operation} on $A$ is a closure operation $c$ on $\insid(A)$ such that $x\cdot I^c\subseteq(xI)^c$ for every $I\in\insid(A)$ \cite{petro}.

\item A \emph{fractional star operation} on an integral domain $D$ is a closure operation $\star$ on $\insfracid(D)$ such that $x\cdot I^\star=(xI)^\star$ for every $I\in\insfracid(D)$ \cite{starloc2}. We denote the set of fractional star operations by $\insfstar(D)$.

\item A \emph{star operation} on an integral domain $D$ is a fractional star operation $\star$ on $\insfracid(D)$ such that $D=D^\star$. We denote the set of star operations by $\insstar(D)$.

\item A \emph{semistar operation} on an integral domain $D$ is a closure operation $\star$ on $\inssubmod_D(K)$ (where $K$ is the quotient field of $D$) such that $x\cdot I^\star=(xI)^\star$ for every $I\in\inssubmod_D(K)$. We denote the set of semistar operations by $\inssemistar(D)$.
\end{itemize}

For the standard results on the theory of star and semistar operations, the reader may consult \cite[Chapter 32]{gilmer}, \cite{okabe-matsuda} or \cite{elliott-libro}.


We will also need the following terminology.
\begin{defin}
Let $(\mathcal{P},\leq)$ be a partially ordered set and let $\mathcal{G}\subseteq\mathcal{P}$. We say that $\mathcal{G}$ is:
\begin{itemize}
\item \emph{downward closed in $\mathcal{P}$} (or a \emph{downset}) if, whenever $x\in \mathcal{G}$ and $y\leq x$, then $y\in \mathcal{G}$;
\item \emph{upward closed in $\mathcal{P}$} (or an \emph{upset}) if, whenever $x\in \mathcal{G}$ and $x\leq y$, then $y\in \mathcal{G}$;
\item an \emph{interval in $\mathcal{P}$} if, whenever $x\leq y$ are in $\mathcal{G}$ and $x\leq z\leq y$, then also $z\in \mathcal{G}$;
\item \emph{quasi-downward closed in $\mathcal{P}$} if $\mathcal{G}$ is an interval with a minimum;
\item \emph{quasi-upward closed in $\mathcal{P}$} if $\mathcal{G}$ is an interval with a maximum.
\end{itemize}
\end{defin}
When $\mathcal{G}$ is a set of $A$-submodules of $B$, we shall silently consider $\mathcal{P}=\inssubmod_A(B)$, and drop the ``in $\mathcal{P}$''.

\section{Multiplicative operations}\label{sect:multiplicative}

\begin{defin}\label{def:main}
Let $A\subseteq B$ be a ring extension, and let $\mathcal{G}\subseteq\inssubmod_A(B)$. A \emph{multiplicative operation} on $(A,B,\mathcal{G})$ is a closure operation $\star:\mathcal{G}\longrightarrow\mathcal{G}$, $I\mapsto I^\star$ such that
\begin{itemize}
\item[$(\dagger)$] $(I:b)^\star\subseteq(I^\star:b)$ for all $I\in\mathcal{G}$, $b\in B$ such that $(I:b)\in\mathcal{G}$.
\end{itemize}
We say that $I\in\mathcal{G}$ is \emph{$\star$-closed} if $I=I^\star$, and we denote by $\mathcal{G}^\star$ the set of $\star$-closed elements of $\mathcal{G}$. We denote by $\insmult(A,B,\mathcal{G})$ the set of those maps.
\end{defin}

Condition $(\dagger)$ is somewhat different from the properties used to define semiprime or semistar operations, since it involves the conductor instead of the product of ideals. However, $(\dagger)$ is more useful in contexts where $\mathcal{G}$ is upward closed, since in this case conductors are in $\mathcal{G}$ more often than products. In the basic cases, this condition reduces to a more usual definition, as we show next.
\begin{lemma}\label{lemma:defproduct}
Let $\mathcal{G}\subseteq\inssubmod_A(B)$ and suppose that, for every $I,J\in\mathcal{G}$, $b\in B$ with $J\neq(0)$ and $b\neq 0$, we have $(I:J),(I:b),IJ,bI\in\mathcal{G}$. Let $c:\mathcal{G}\longrightarrow\mathcal{G}$ be a closure operation. The following are equivalent:
\begin{enumerate}[(i)]
\item\label{lemma:defproduct:prodb} $bI^\star\subseteq(bI)^\star$ for all $I\in\inssubmod_A(B)$ and all $b\in B$;
\item\label{lemma:defproduct:prodJ} $JI^\star\subseteq(JI)^\star$ for all $I,J\in\inssubmod_A(B)$;
\item\label{lemma:defproduct:colon} $(I:b)^\star\subseteq(I^\star:b)$ for all $I\in\inssubmod_A(B)$ and all $b\in B$.
\end{enumerate}
\end{lemma}
\begin{proof}
The equivalence of \ref{lemma:defproduct:prodb} and \ref{lemma:defproduct:prodJ} is proved in \cite[Chapter 3, Lemma 3.1]{knebush-II}. Suppose \ref{lemma:defproduct:prodb} holds: then, using it on $(I^\star:b)$, we have 
\begin{equation*}
b(I^\star:b)^\star\subseteq(b(I^\star:b))^\star\subseteq (I^\star)^\star=I^\star,
\end{equation*}
and so \ref{lemma:defproduct:colon} holds. Conversely, if \ref{lemma:defproduct:colon} holds then
\begin{equation*}
I^\star\subseteq(bI:b)^\star\subseteq((bI)^\star:b)
\end{equation*}
and so $bI^\star\subseteq(bI)^\star$, i.e., \ref{lemma:defproduct:prodb} holds. The claim is proved.
\end{proof}

The hypothesis of the lemma hold, for example, if $\mathcal{G}=\inssubmod_A(B)$ or if $A$ is a domain, $B$ its quotient field and $\mathcal{G}$ the set of fractional ideals of $A$.

\begin{ex}\label{ex:general}
Definition \ref{def:main} is very general, and includes several already studied classes of closure operations.
\begin{enumerate}[(1)]
\item\label{ex:general:extension} If $\mathcal{G}=\inssubmod_A(B)$, then the notion of multiplicative operation on $(A,B,\mathcal{G})$ coincides with the notion of \emph{star operation on the extension $A\subseteq B$}, as defined in \cite[Chapter 3]{knebush-II}.
\item\label{ex:general:semistar} If $A$ is a domain and $B$ is its quotient field, then a multiplicative operation on $(A,B,\inssubmod_A(B)^\nz)$ is just a semistar operation on $A$. On the other hand, the multiplicative operations on $(A,B,\inssubmod_A(B))$ are the semistar operations and the closure sending every submodule to $B$.
\item\label{ex:general:semiprime} If $A=B$, then $\inssubmod_A(A)=\insid(A)$ is just the set of ideals of $A$. By Lemma \ref{lemma:defproduct}, the multiplicative operations on $(A,A,\insid(A))$ are exactly the semiprime operations on $A$.
\item\label{ex:general:star} If $A$ is a domain and $B$ is its quotient field, then a multiplicative operation on $(A,B,\insfracid(A)^\nz)$ is exactly a fractional star operation on $A$. To obtain star operations, we need to consider only integral ideals; while the proof is easy, it needs some care, and for this reason we give it explicitly in the next proposition.
\end{enumerate}
\end{ex}

\begin{prop}\label{prop:ex-star}
Let $D$ be an integral domain and let $K$ be its quotient field. Then, the map
\begin{equation*}
\begin{aligned}
\Psi\colon\insstar(D) & \longrightarrow\insmult(D,K,\insid(D)^\nz),\\
\star & \longmapsto \star|_{\insid(D)^\nz}
\end{aligned}
\end{equation*}
is an order isomorphism.
\end{prop}
\begin{proof}
The restriction of a star operation is clearly multiplicative, and thus $\Psi$ is well-defined and order-preserving. Furthermore, if $\star_1\neq\star_2$ then $I^{\star_1}\neq I^{\star_2}$ for some fractional ideal $I$; if $d\neq 0$ satisfies $dI\subseteq D$, then $(dI)^{\star_1}=dI^{\star_1}\neq dI^{\star_2}=(dI)^{\star_2}$, and thus $\Psi$ is also injective.

Suppose now that $\ast\in\insmult(D,K,\insid(D)^\nz)$. For every fractional ideal $I$ of $D$, let $d\neq 0$ be such that $dI\subseteq D$, and define $I^\star:=d^{-1}(dI)^\ast$. We claim that $\star$ is a star operation whose restriction is $\ast$.

We first show that $\star$ is well-defined: indeed, suppose $dI,eI\subseteq D$ for some $d,e\neq 0$. Then, $(dI:e^{-1}d)=eI\subseteq D$, and thus
\begin{equation*}
(eI)^\ast=(dI:e^{-1}d)^\ast\subseteq((dI)^\ast:e^{-1}d)=ed^{-1}(dI)^\ast
\end{equation*}
and so $e^{-1}(eI)^\ast\subseteq d^{-1}(dI)^\ast$. By symmetry, also the opposite containment holds, and thus $\ast$ is well-defined. The definition immediately implies that $\star$ is extensive; it is also order-preserving since if $I\subseteq J$ and $dJ\subseteq D$ then also $dI\subseteq D$. It is idempotent since, if $dI\subseteq D$, then $dI^\star=d(d^{-1}(dI)^\ast)=(dI)^\ast$ and so
\begin{equation*}
(I^\star)^\star=d^{-1}(dI)^\ast=d^{-1}dI^\star=I^\star.
\end{equation*}
Similarly, for every $e\in K$, if $dI\subseteq D$ then $e^{-1}deI\subseteq D$ and thus
\begin{equation*}
(eI)^\star=ed^{-1}(e^{-1}deI)^\ast=e(d^{-1}(dI)^\ast))=eI^\star.
\end{equation*}
Finally, using $d=1$ (and the fact that $D$ is the maximum of $\insid(D)^\nz$) we see that $D^\star=D$ and $I^\star=I^\ast$ for every $I\in\insid(D)^\nz$; therefore, $\ast=\Psi(\star)$ and $\Psi$ is surjective. Hence, it is an order isomorphism, as claimed.
\end{proof}

Like in the case of semistar operations, using $\insid(D)$ instead of $\insid(D)^\nz$ we obtain a single new multiplicative operation, namely the one sending every ideal to $A$.

\begin{oss}
~\begin{enumerate}[(1)]
\item As an $A$-algebra, $B$ is also an $A$-module. There are several kinds of closure operations that can be defined on the submodules of an $A$-module (see e.g. \cite[Section 7]{epstein_guide}); however, the notion of multiplicative operation cannot be reduced to any of them, since multiplication by elements of $B$ is an integral part of Definition \ref{def:main}. For example, if $A$ is a field and $\mathcal{G}=\inssubmod_A(B)$, the set of multiplicative operations on $(A,B,\mathcal{G})$ depends not only on the $A$-module structure of $B$ (i.e., on the dimension of $B$ over $A$) but also on its ring structure.
\item Even if all elements of $\mathcal{G}$ are contained in a subalgebra $B'\subseteq B$, the multiplicative operations on $(A,B,\mathcal{G})$ are not the same as the multiplicative operations on $(A,B',\mathcal{G})$. For example, if $A$ is a domain and $\mathcal{G}$ are the ideals of $A$, the the multiplicative operations on $(A,A,\mathcal{G})$ are the semiprime operations, while the multiplicative operations on $(A,K,\mathcal{G})$ (where $K$ is the quotient field of $A$) correspond, by Proposition \ref{prop:ex-star}, to star operations (plus the trivial map $I\mapsto A$).
\end{enumerate}
\end{oss}

While multiplicative operations usually act very similarly to semistar, star or semiprime operations, often additional care is needed in describing their properties, since the set $\mathcal{G}$ of definition may not be closed by sums or by intersections. We usually circumvent this problem by adding some hypothesis on $\mathcal{G}$.

\begin{prop}\label{prop:prop}
Let $\star$ be a multiplicative operation on $(A,B,\mathcal{G})$, and let $I\in\mathcal{G}$, $\{I_\lambda\}_{\lambda\in\Lambda}\subseteq\mathcal{G}$. Then, the following hold.
\begin{enumerate}[(a),itemsep=3pt]
\item\label{prop:prop:sum} If $\sum_\lambda I_\lambda,\sum_\lambda I_\lambda^\star\in\mathcal{G}$, then $\left(\sum_\lambda I_\lambda\right)^\star=\left(\sum_\lambda I_\lambda^\star\right)^\star$. In particular, this happens if $\mathcal{G}$ is quasi-upward closed.
\item\label{prop:prop:intersec} If $\bigcap_\lambda I_\lambda,\bigcap_\lambda I_\lambda^\star\in\mathcal{G}$, then $\left(\bigcap_\lambda I_\lambda\right)^\star=\left(\bigcap_\lambda I_\lambda^\star\right)^\star$. In particular, this happens if $\mathcal{G}$ is quasi-downward closed.
\item\label{prop:prop:colon} If $I$ is $\star$-closed and $(I:b)\in\mathcal{G}$, then $(I:b)$ is $\star$-closed.
\item\label{prop:prop:Istar} $I^\star$ is equal to the intersection of all the $\star$-closed modules containing $I$.
\end{enumerate}
\end{prop}
\begin{proof}
It is enough to apply the definitions.
\end{proof}

Point \ref{prop:prop:Istar} of the previous proposition has a significant consequence.
\begin{cor}\label{cor:ugclosed}
Let $\star_1,\star_2\in\insmult(A,B,\mathcal{G})$. If $\mathcal{G}^{\star_1}=\mathcal{G}^{\star_2}$, then $\star_1=\star_2$.
\end{cor}
\begin{proof}
By Proposition \ref{prop:prop}\ref{prop:prop:Istar}, for every $I\in\mathcal{G}$ the closure $I^\star$ can be defined uniquely through $\mathcal{G}^\star$. Hence, if $\star_1$ and $\star_2$ close exactly the same modules then they must be equal.
\end{proof}

\medskip

Two common properties of closure operations in the algebraic setting are being of finite type and being stable. These definitions and the related constructions works only partially for multiplicative operations.

\begin{prop}
Let $A\subseteq B$ be a ring extension and let $\mathcal{G}\subseteq\inssubmod_A(B)$ be downward closed; let $\star\in\insmult(A,B,\mathcal{G})$. The map
\begin{equation*}
\star_f:I\mapsto\bigcup\{J^\star\mid J\subseteq I\text{~is finitely generated}\}
\end{equation*}
is a well-defined multiplicative operation on $(A,B,\mathcal{G})$, and $\star_f=(\star_f)_f$.
\end{prop}
We say that $\star$ is \emph{of finite type} if $\star=\star_f$.
\begin{proof}
Since $\mathcal{G}$ is downward closed, $J^\star$ is defined for every finitely generated $A$-module $J\subseteq I$; furthermore, the union is in $\mathcal{G}$ because it is contained in $I^\star$. Hence, $\star_f$ is well-defined as a map from $\mathcal{G}$ to $\mathcal{G}$. The fact that $\star_f$ is a closure operation follows exactly as in the star or in the semiprime case.

Let $I\in\mathcal{G}$ and $b\in B$. Suppose $x\in(I:b)^{\star_f}$: then, there is a finitely generated $J\subseteq(I:b)$ such that $x\in J^\star$. Hence,
\begin{equation*}
x\in J^\star\subseteq(bJ:b)^\star\subseteq((bJ)^\star:b)\subseteq(I^{\star_f}:b)
\end{equation*}
since $bJ\subseteq I$ is finitely generated and thus $(bJ)^\star=(bJ)^{\star_f}\subseteq I^{\star_f}$. Hence, $\star_f$ is multiplicative; the fact that $\star_f=(\star_f)_f$ follows again as in the star or in the semiprime setting.
\end{proof}

We say that a multiplicative operation $\star$ on $(A,B,\mathcal{G})$ is \emph{stable} if $(I\cap J)^\star=I^\star\cap J^\star$ for every $I,J\in\mathcal{G}$.
\begin{prop}
Let $A\subseteq B$ be a ring extension and let $\mathcal{G}\subseteq\inssubmod_A(B)$ be downward closed; suppose $A\in\mathcal{G}$.
\begin{enumerate}[(a)]
\item The map
\begin{equation*}
\overline{\star}:I\mapsto\bigcup\{(I:E)\mid E^\star=A^\star\}
\end{equation*}
is a well-defined stable multiplicative operation on $(A,B,\mathcal{G})$. Furthermore, $\overline{(\overline{\star})}=\overline{\star}$, and $\star$ is stable if and only if $\star=\overline{\star}$.
\item The map
\begin{equation*}
\star_w:I\mapsto\bigcup\{(I:E)\mid E^\star=A^\star,~E\text{~is finitely generated}\}
\end{equation*}
is a well-defined stable multiplicative operation of finite type on $(A,B,\mathcal{G})$. Furthermore, $\star_w=(\star_w)_w$.
\end{enumerate}
\end{prop}
\begin{proof}
Since $A\in\mathcal{G}$, it makes sense to consider the modules $E$ such that $E^\star=A^\star$. If $x\in(I:E)$ and $E^\star=A^\star$, then $E\subseteq(I:x)$ and thus
\begin{equation*}
1\in A^\star=E^\star\subseteq(I:x)^\star\subseteq(I^\star:x),
\end{equation*}
so $x\in I^\star$, i.e., $(I:E)\subseteq I^\star$. Hence, $I^{\overline{\star}}\subseteq I^\star$ and since $\mathcal{G}$ is downward closed $I^{\overline{\star}}\in\mathcal{G}$. The fact that $\overline{\star}$ is a closure operations follows as in the star operation setting (see e.g. \cite{anderson_two_2000}).

Let $I\in\mathcal{G}$ and $b\in B$. If $x\in(I:b)^{\overline{\star}}$, then there is an $E$ such that $E^\star=A^\star$ and such that $x\in((I:b):E)$. Thus, $xbE\subseteq I$ and $xb\in(I:E)\subseteq I^{\overline{\star}}$. Hence, $x\in(I^{\overline{\star}}:b)$ and $(I:b)^{\overline{\star}}\subseteq(I^{\overline{\star}}:b)$, i.e., $\overline{\star}$ is multiplicative.

All the other claims follow as in the star operation setting (see again \cite{anderson_two_2000}). The case for $\star_w$ is completely analogous (adding ``finitely generated'' when needed).
\end{proof}

\medskip

The set $\insmult(A,B,\mathcal{G})$ inherits from $\insclos(\mathcal{G})$ a natural order structure. Unlike $\insstar(D)$ or $\inssemistar(D)$, in general this order does not satisfy any significant property: for example, it need not to admit infima or suprema, even of finite subsets. Adding hypothesis on $\mathcal{G}$ helps.
\begin{prop}\label{prop:infsup}
Let $A\subseteq B$ be a ring extension and let $\mathcal{G}\subseteq\inssubmod_A(B)$ be an interval. Let $\Lambda\subseteq\insmult(A,B,\mathcal{G})$ be nonempty. Then, the following hold.
\begin{enumerate}[(a)]
\item\label{prop:infsup:inf} $\Lambda$ has an infimum $\inf\Lambda$, and
\begin{equation*}
I^{\inf\Lambda}=\bigcap_{\star\in\Lambda}I^\star
\end{equation*}
for every $I\in\mathcal{G}$.
\item\label{prop:infsup:sup} Suppose that every $L\in\mathcal{G}$ is contained in a submodule that is $\star$-closed for every $\star\in\Lambda$. Then, $\Lambda$ has a supremum $\sup\Lambda$, and
\begin{equation*}
I^{\sup\Lambda}=\bigcap\{J\in\mathcal{G}\mid I\subseteq J=J^\star\text{~for all~}\star\in\Lambda\}
\end{equation*}
for every $I\in\mathcal{G}$. In particular, if $\mathcal{G}$ is quasi-upward closed then $\sup\Lambda$ exists for every $\Lambda\subseteq\insmult(A,B,\mathcal{G})$.
\end{enumerate}
\end{prop}
\begin{proof}
\ref{prop:infsup:inf} Let $\sharp$ be the map sending $I$ to $\bigcap_{\star\in\Lambda}I^\star$. If $\star\in\Lambda$ is arbitrary, then since $\mathcal{G}$ is an interval every element between $I$ and $I^\star$ is in $\mathcal{G}$; this applies in particular to the intersection $\bigcap_{\star\in\Lambda}I^\star$, so that $\sharp$ is actually a map from $\mathcal{G}$ to itself.

Clearly, $\sharp$ is a closure operation, and it is the infimum of $\Lambda$ in the set of all closures. Suppose $I,(I:b)\in\mathcal{G}$. Then,
\begin{equation*}
(I:b)^\sharp=\bigcap_{\star\in\Lambda}(I:b)^\star\subseteq\left(\bigcap_{\star\in\Lambda}I^\star:b\right)=(I^\sharp:b)
\end{equation*}
and thus $\sharp$ is multiplicative. Hence, it is the infimum of $\Lambda$ in $\insmult(A,B,\mathcal{G})$.

\ref{prop:infsup:sup} Let $I^\sharp$ be the intersection on the right hand side; the hypothesis guarantees that there is always some $J=J^\star$ containing $I$, and thus $I^\sharp$ is well-defined. Since $\mathcal{G}$ is an interval and $I\subseteq I^\sharp$, we see that $I^\sharp\in\mathcal{G}$; hence, the map $\sharp:I\mapsto I^\sharp$ is a self-map of $\mathcal{G}$. It is also easy to see that $\sharp$ is a closure operation and that it is the supremum of $\Lambda$ in $\insclos(A,B,\mathcal{G})$.

Suppose $I,(I:b)\in\mathcal{G}$. The module $I^\sharp$ is $\star$-closed for every $\star\in\Lambda$ and, since each $\star$ is multiplicative, also $(I^\sharp:b)$ is $\star$-closed; it follows that $(I:b)^\sharp\subseteq(I^\sharp:b)$, so that $\sharp$ is multiplicative.

The last remark follows, since the maximum of $\mathcal{G}$ is closed by every multiplicative operation.
\end{proof}

\begin{oss}
~\begin{enumerate}[(1)]
\item Part \ref{prop:infsup:sup} can be applied, in particular, if $B\in\mathcal{G}$; hence, it holds in the set of semiprime or semistar operations.
\item Let $A$ be a domain and $B$ its quotient field. If $\mathcal{G}$ is the set of nonzero fractional ideals and $\Lambda$ is the set of star operations on $A$, then part \ref{prop:infsup:sup} can be applied, since every fractional ideal $J$ is contained in some principal ideal $bA$ and $(bA)^\star=bA^\star=bA$ for every $b\in B$. This can also be seen by viewing star operations as the multiplicative operations on $(A,B,\insid(A)^\nz)$ (through Proposition \ref{prop:ex-star}) since $\insid(A)^\nz$ has a maximum ($A$ itself).
\item Let $A,B,\mathcal{G}$ be as in the previous case, and suppose that the set $\mathcal{T}$ of overrings of $A$ that are fractional ideals of $A$ has no maximum (or, equivalently, no maximal element). Associate to each $T\in\mathcal{T}$ the multiplicative operation $\wedge_T:I\mapsto IT$: then, the set $\{\wedge_T\mid T\in\mathcal{T}\}$ has no supremum in $\insmult(A,B,\mathcal{G})$. This happens, for example, if $A$ is a valuation domain without an height-one prime (as all localizations of $A$ different from the quotient field $B$ are fractional ideals) or if $A$ is a one-dimensional Noetherian domain whose integral closure is not finite over $A$.
\end{enumerate}
\end{oss}

Let $A\subseteq B$ be a ring extension, and let $B_1,B_2$ be two rings between $A$ and $B$. Let $\mathcal{G}\subseteq\inssubmod_A(B_1)\cap\inssubmod_A(B_2)$ be an interval. While multiplicative operations on $(A,B_1,\mathcal{G})$ and $(A,B_2,\mathcal{G})$ do not necessarily coincide (see Example \ref{ex:general}, points \ref{ex:general:semiprime} and \ref{ex:general:star}), we can consider both $\insmult(A,B_1,\mathcal{G})$ and $\insmult(A,B_2,\mathcal{G})$ as subsets of $\insclos(\mathcal{G})$, the set of closures on $\mathcal{G}$, so in particular it makes sense to consider their intersection. The constructions in Proposition \ref{prop:infsup} imply the following.
\begin{prop}\label{prop:infsupB1B2}
Let $A,B,B_1,B_2,\mathcal{G}$ be as above, and suppose that $\mathcal{G}$ is an interval. Let $\Lambda\subseteq\insmult(A,B_1,\mathcal{G})\cap\insmult(A,B_2,\mathcal{G})$. Then, the following hold.
\begin{enumerate}[(a)]
\item The infimum of $\Lambda$ inside $\insmult(A,B_1,\mathcal{G})$ coincides with the infimum of $\Lambda$ in $\insclos(\mathcal{G})$ and in $\insmult(A,B_2,\mathcal{G})$.
\item Suppose $\Lambda$ satisfies the hypothesis of Proposition \ref{prop:infsup}\ref{prop:infsup:sup}. Then, the supremum of $\Lambda$ inside $\insmult(A,B_1,\mathcal{G})$ coincides with the supremum of $\Lambda$ in $\insclos(\mathcal{G})$ and in $\insmult(A,B_2,\mathcal{G})$.
\end{enumerate}
\end{prop}
We shall study in a deeper way the relationship between $\insmult(A,B_1,\mathcal{G})$ and $\insmult(A,B_2,\mathcal{G})$ in Section \ref{sect:functoriality}.

\section{Changing $\mathcal{G}$ and principal operations}\label{sect:Gprinc}
In this section, we want to analyze how the set of multiplicative operations changes by going from $\mathcal{G}_1$ to $\mathcal{G}_2$ (and conversely), where $\mathcal{G}_1\subseteq\mathcal{G}_2$ are two subsets of $\inssubmod_A(B)$ (with $A$ and $B$ remaining fixed). The passage from $\mathcal{G}_2$ to $\mathcal{G}_1$ can be obtained by restriction.
\begin{prop}\label{prop:restriction}
Let $A\subseteq B$ be a ring extension, and let $\mathcal{G}_1\subseteq\mathcal{G}_2$ be two subsets of $\inssubmod_A(B)$.
\begin{enumerate}[(a)]
\item\label{prop:restriction:ex} If $\star\in\insmult(A,B,\mathcal{G}_2)$ restricts to a map $\star':\mathcal{G}_1\longrightarrow\mathcal{G}_1$, then $\star'\in\insmult(A,B,\mathcal{G}_1)$.
\item\label{prop:restriction:map} If $\mathcal{G}_1$ is an interval and $\mathcal{G}_1$ and $\mathcal{G}_2$ have a common maximum, then every multiplicative operation on $(A,B,\mathcal{G}_2)$ can be restricted to $\mathcal{G}_1$, and the map
\begin{equation*}
\begin{aligned}
\rho\colon\insmult(A,B,\mathcal{G}_2) & \longrightarrow\insmult(A,B,\mathcal{G}_1)\\
\star & \longmapsto\star|_{\mathcal{G}_1}
\end{aligned}
\end{equation*}
is well-defined and order-preserving.
\end{enumerate}
\end{prop}
\begin{proof}
Part \ref{prop:restriction:ex} is obvious. For part \ref{prop:restriction:map}, let $T$ be the maximum of $\mathcal{G}_1$ and $\mathcal{G}_2$. Then, $T^\star=T$ for every $\star\in\insmult(A,B,\mathcal{G}_2)$, and thus $I^\star\subseteq T$ for every $I\in\mathcal{G}_1$. Since $\mathcal{G}_1$ is an interval, it follows that $I^\star\in\mathcal{G}_1$ for every $I\in\mathcal{G}_1$, and thus $\star$ restricts to a multiplicative operation on $\mathcal{G}_1$. Hence, $\rho$ is well-defined; the fact that it is order-preserving follows immediately from the definitions.
\end{proof}

To study extensions from $\mathcal{G}_1$ to $\mathcal{G}_2$, we need a way to build closures: we do so by attaching to each element of $\mathcal{G}$ a multiplicative operation. We shall give the definition in the quasi-upward closed case, and prove the extension property in the upward closed case; the latter condition, in particular, does not cover all important cases (in particular, it misses star operations), but will be enough for the applications in Section \ref{sect:Artin}.

\begin{defin}
Suppose $\mathcal{G}$ is quasi-upward closed, and let $J\in\mathcal{G}$. The multiplicative operation \emph{generated by $J$} on $(A,B,\mathcal{G})$ is the closure
\begin{equation*}
\princ_J:=\sup\{\star\in\insmult(A,B,\mathcal{G})\mid J=J^\star\}.
\end{equation*}
Likewise, if $\mathcal{S}\subseteq\mathcal{G}$, the multiplicative operation \emph{generated by $\mathcal{S}$} is
\begin{equation*}
\princ_\mathcal{S}:=\sup\{\star\in\insmult(A,B,\mathcal{G})\mid J=J^\star\text{~for all~}J\in\mathcal{S}\}=\inf\{\princ_J\mid J\in\mathcal{S}\}
\end{equation*}
If $\star=\princ_J$ for some $J\in\mathcal{G}$, we say that $\star$ is a \emph{principal} multiplicative operation.
\end{defin}

Note that the supremum is well-defined since $\mathcal{G}$ has a maximum, and so the supremum exist (and belongs to $\insmult(A,B,\mathcal{G})$) by Proposition \ref{prop:infsup}\ref{prop:infsup:sup}. 
\begin{lemma}\label{lemma:chiusi}
Suppose $\mathcal{G}$ is quasi-upward closed.  Then, $\star=\star_{\mathcal{G}^\star}$ for every $\star\in\insmult(A,B,\mathcal{G})$.
\end{lemma}
\begin{proof}
Let $\mathcal{S}:=\mathcal{G}^\star$. If $I\in\mathcal{S}$, then $\princ_I\geq\star$, and thus $\princ_{\mathcal{S}}\geq\star$. In particular, $\mathcal{G}^{\princ_{\mathcal{S}}}\subseteq\mathcal{G}^\star=\mathcal{S}$. Furthermore, $I$ is closed by $\princ_{\mathcal{S}}$, and thus $\mathcal{G}^\star=\mathcal{S}\subseteq\mathcal{G}^{\princ_{\mathcal{S}}}$. Hence, $\mathcal{S}=\mathcal{G}^{\princ_{\mathcal{S}}}$ and thus $\star=\princ_{\mathcal{S}}$ by Corollary \ref{cor:ugclosed}.
\end{proof}

If $\mathcal{G}$ is upward closed, we can give a more explicit description of $\princ_J$ by means of a ``double dual'' construction.
\begin{lemma}\label{lemma:div}
Suppose that $\mathcal{G}$ is upward closed in $\inssubmod_A(B)$. Then, $I^{\princ_J}=(J:_B(J:_BI))$ for all $I\in\mathcal{G}$.
\end{lemma}
\begin{proof}
For simplicity of notation, we use $(I:J)$ for $(I:_BJ)$.

We first show that
\begin{equation*}
(J:(J:I))=\bigcap\{(J:b)\mid I\subseteq(J:b)\}.
\end{equation*}
Indeed, if $x\in(J:(J:I))$ then $x(J:I)\subseteq J$. If $I\subseteq(J:b)$, then $b\in(J:I)$, and thus $xb\in J$; hence, $x\in(J:b)$ and thus $x$ is in the intersection. Conversely, if $x$ is in the intersection and $s\in(J:I)$, then $I\subseteq(J:s)$ and thus $x\in(J:s)$, i.e., $xs\in J$; since $s$ was arbitrary, $x(J:I)\subseteq J$, or equivalently $x\in(J:(J:I))$.

Hence, the map $\sharp$ sending $I$ to $(J:(J:I))$ is a closure. With a reasoning analogous to the one in \cite[Chapter 3, Proposition 3.6]{knebush-II} or \cite[Proposition 3.2]{hhp_m-canonical}, we see that it satisfies $b\cdot I^\sharp\subseteq (bI)^\sharp$ for every $I\in\inssubmod_A(B)$ and every $b$, and thus it is multiplicative on $(A,B,\inssubmod_A(B))$ by Lemma \ref{lemma:defproduct}; hence, it is also multiplicative on $(A,B,\mathcal{G})$. Since $J^\sharp=(J:(J:J))=J$, we have $\princ_J\leq\sharp$; conversely, since $J$ is $\princ_J$-closed then also each $(J:b)$ is $\princ_J$-closed, and so also $(J:(J:I))$ is $\princ_J$-closed, being the intersection of $\princ_J$-closed modules. Hence, $I^{\princ_J}\subseteq(J:(J:I))$ and so $\princ_J\leq\sharp$. Therefore, $\princ_J=\sharp$, as claimed.
\end{proof}

As defined, the multiplicative operation $\princ_J$ is relative to the triple $(A,B,\mathcal{G})$. However, Lemma \ref{lemma:div} shows that it is essentially independent from $\mathcal{G}$.
\begin{cor}\label{cor:princ-ug}
Let $\mathcal{G}_1\subseteq\mathcal{G}_2$ be upward closed subsets, and let $J\in\mathcal{G}_1$; for $i\in\{1,2\}$, let $\star_i$ be the principal operation generated by $J$ on $(A,B,\mathcal{G}_i)$. Then, $\star_2|_{\mathcal{G}_1}=\star_1$.
\end{cor}
\begin{proof}
By Lemma \ref{lemma:div}, for every $I\in\mathcal{G}_1$ we have $I^{\star_2}=(J:_B(J:_BI))=I^{\star_1}$. The claim is proved.
\end{proof}

Lemma \ref{lemma:div} does not hold if $\mathcal{G}$ is only quasi-upward closed. For example, let $A$ be an integral domain, $B$ its quotient field and $\mathcal{G}:=\insid(A)$. If $J\in\mathcal{G}$ satisfies $(J:J)\neq A$, then the map $I\mapsto(J:(J:I))$ sends $A$ to $(J:J)\notin\mathcal{G}$, and thus it is not equal to the principal operation generated by $J$. Furthermore, if $\mathcal{G}':=\insfracid(A)$, then with the same notation the multiplicative operation generated by $J$ on $(A,B,\mathcal{G})$ will be different from the multiplicative operation generated by $J$ on $(A,B,\mathcal{G}')$.

\begin{prop}\label{prop:extension}
Let $A\subseteq B$ be a ring extension, and let $\mathcal{G}_1\subseteq\mathcal{G}_2$ be two upward closed subsets of $\inssubmod_A(B)$. Then, the following hold.
\begin{enumerate}[(a)]
\item\label{prop:extension:ext} For every $\star\in\insmult(A,B,\mathcal{G}_1)$, the map
\begin{equation*}
\begin{aligned}
\widehat{\star}\colon\mathcal{G}_2 & \longrightarrow\mathcal{G}_2\\
I & \longmapsto\bigcap_{\substack{J\in\mathcal{G}_1\\ J=J^\star}}(J:_B(J:_BI))
\end{aligned}
\end{equation*}
is a multiplicative operation on $(A,B,\mathcal{G}_2)$.
\item\label{prop:extension:map} The map
\begin{equation*}
\begin{aligned}
\eta\colon\insmult(A,B,\mathcal{G}_1) & \longrightarrow\insmult(A,B,\mathcal{G}_2)\\
\star & \longmapsto\widehat{\star}
\end{aligned}
\end{equation*}
is well-defined and order-preserving.
\item\label{prop:extension:inv} If $\rho$ is the map defined in Proposition \ref{prop:restriction}, then $\rho\circ\eta$ is the identity on $\insmult(A,B,\mathcal{G}_1)$, i.e., $\widehat{\star}|_{\mathcal{G}_1}=\star$. In particular, $\eta$ is injective and $\rho$ is surjective.
\end{enumerate}
\end{prop}
\begin{proof}
By Lemma \ref{lemma:div} and Corollary \ref{cor:princ-ug}, $\widehat{\star}$ is just the multiplicative operation generated by $\mathcal{G}_1^\star:=\{J\in\mathcal{G}_1\mid J=J^\star\}$ on $\mathcal{G}_2$, so that $\eta:\star\longmapsto\widehat{\star}$ is well-defined; the fact that $\eta$ is order-preserving is obvious.

Furthermore, by Lemma \ref{lemma:chiusi}, $\star$ can be seen as the multiplicative operation generated by $\mathcal{G}_1^\star$ in $\mathcal{G}_1$: hence, if $I\in\mathcal{G}_1$ then $I^\star=\bigcap\{(J:(J:I))\mid J\in\mathcal{G}_1^\star\}$, and in particular $I^{\star}=I^{\widehat{\star}}$. Thus $\widehat{\star}$ restricts to $\star$ and $\rho\circ\eta$ is the identity.
\end{proof}

In general, the extension of a multiplicative operation $\star$ on $(A,B,\mathcal{G})$ is not unique, either for trivial (for example, if $\mathcal{G}_1=\{B\}$) or nontrivial reason (see Example \ref{ex:dvr-sp} below).

We now use principal operations to introduce an order on $\mathcal{G}$.
\begin{defin}
Let $A\subseteq B$ be a ring extension, and let $\mathcal{G}\subseteq\inssubmod_A(B)$ be quasi-upward closed. Let $I,J\in\mathcal{G}$. We say that $I$ is \emph{multiplicatively minor than $J$ on $(A,B,\mathcal{G})$} if $\princ_I\geq\princ_J$ (where $\princ_I$ denotes the principal operation generated by $I$ on $(A,B,\mathcal{G})$), or equivalently if $J$ is $\princ_I$-closed. In this case, we write $I\ordine_{(A,B,\mathcal{G})} J$, or $I\ordine J$ if the triple $(A,B,\mathcal{G})$ is understood from the context. We call $\ordine$ the \emph{multiplicative order} of $\mathcal{G}$ with respect to $A\subseteq B$.
\end{defin}

Note that, like the operation generated by a module, the relation $\ordine$ is relative to the triple $(A,B,\mathcal{G})$. Corollary \ref{cor:princ-ug} can be naturally reparaphrased using $\ordine$.
\begin{prop}
Let $\mathcal{G}_1\subseteq\mathcal{G}_2$ be upward closed, and let $I,J\in\mathcal{G}_1$. Then, $I\ordine_{(A,B,\mathcal{G}_1)} J$ if and only if $I\ordine_{(A,B,\mathcal{G}_2)} J$.
\end{prop}
\begin{proof}
Both conditions are equivalent to $J=(I:_B(I:_BJ))$.
\end{proof}

The relation $\ordine=\ordine_{(A,B,\mathcal{G})}$ is easily seen to be reflexive and transitive, and thus a preorder. In general, it is not an order: for example, if $A$ is a domain, $B$ its quotient field and $\mathcal{G}=\inssubmod_A(B)$, then $I\ordine bI$ and $bI\ordine I$ for every nonzero $b\in B$. We denote by $[\mathcal{G}]$ the set of equivalence classes of $\mathcal{G}$ through the equivalence relation induced by $\ordine$ on $\mathcal{G}$ (i.e., such that $I$ and $J$ are equivalent if and only if $I\ordine J$ and $J\ordine I$), and by $[I]$ the class of $I\in\mathcal{G}$. For ease of notation, we continue to use $\ordine$ to denote the order induced on $[\mathcal{G}]$ by the preorder on $\mathcal{G}$.

\begin{prop}\label{prop:ordine}
Let $\mathcal{G}$ be quasi-upward closed. Then, the following hold.
\begin{enumerate}[(a)]
\item\label{prop:ordine:min} If $T$ is the maximum of $\mathcal{G}$, then $[T]$ is the minimum of $([\mathcal{G}],\ordine)$.
\item\label{prop:ordine:downward} The set $\mathcal{G}^\star$ is downward closed in $(\mathcal{G},\ordine)$.
\item\label{prop:ordine:princ} For every $J\in\mathcal{G}$, the set $\mathcal{G}^{\princ_J}$ is the down-set of $J$ in $(\mathcal{G},\ordine)$.
\item\label{prop:ordine:max} If $([\mathcal{G}],\ordine)$ has a maximum $[\omega]$, then $\princ_\omega$ is the identity.
\end{enumerate}
\end{prop}
\begin{proof}
\ref{prop:ordine:min} follows from the fact that $T$ is $\star$-closed for every $\star\in\insmult(A,B,\mathcal{G})$ (as $T\subseteq T^\star$). \ref{prop:ordine:downward} and \ref{prop:ordine:princ} follow directly from the definitions, while \ref{prop:ordine:max} is a direct consequence of \ref{prop:ordine:princ}.
\end{proof}

Note that $([\mathcal{G}],\ordine)$ may not have a maximum, or even maximal elements. If $[\omega]$ is a maximum of $([\mathcal{G}],\ordine)$, we say that $\omega$ is a \emph{canonical ideal} for $(A,B,\mathcal{G})$.

Since $\mathcal{G}^\star$ is downward closed, if we want to describe all multiplicative operations on $(A,B,\mathcal{G})$ we just have to find which downsets are in the form $\mathcal{G}^\star$. The following is a useful criterion.
\begin{prop}
Let $\mathcal{G}$ be upward closed, and let $\mathcal{D}$ be a downset of $(\mathcal{G},\ordine)$. Then, $\mathcal{D}=\mathcal{G}^\star$ for some $\star\in\insmult(A,B,\mathcal{G})$ if and only if the intersection of every subfamily of $\mathcal{D}$ is either in $\mathcal{D}$ or out of $\mathcal{G}$.
\end{prop}
\begin{proof}
Let $I$ be the intersection of a family of $\star$-closed ideals. If $I\in\mathcal{G}$, then $I$ must be $\star$-closed; hence, if $\mathcal{D}=\mathcal{G}^\star$ and $I\in\mathcal{G}$ then $I\in\mathcal{D}$.

Conversely, suppose the condition holds, and define, for every $J\in\mathcal{G}$,
\begin{equation*}
J^\star:=\bigcap\{I\in\mathcal{D}\mid J\subseteq I\}.
\end{equation*}
Since $\mathcal{G}$ is upward closed, the map $J\mapsto J^\star$ goes from $\mathcal{G}$ to $\mathcal{G}$ and is a closure operation. To show that it is multiplicative, let $J\in\mathcal{G}$ and $b\in B$ be such that $(J:b)\in\mathcal{G}$. Then,
\begin{equation*}
(J^\star:b)=\left(\bigcap_{\substack{L\in\mathcal{D}\\ J\subseteq L}}L:b\right)=\bigcap_{\substack{L\in\mathcal{D}\\ J\subseteq L}}(L:b).
\end{equation*}
For all these $L$, we have $(J:b)\subseteq(L:b)$, and thus $(L:b)\in\mathcal{G}$; since $(L:b)\ordine L$ and $\mathcal{D}$ is a downset with respect to $\ordine$, we have $(L:b)\in\mathcal{D}$. Hence, $(J:b)^\star\subseteq(L:b)$ for every such $L$, and thus $(J:b)^\star$ is contained in the intersection. Therefore, $(J:b)^\star\subseteq(J^\star:b)$, and $\star$ is multiplicative. It follows that $\mathcal{G}^\star=\mathcal{D}$.
\end{proof}

To conclude this section, we show how to use Lemma \ref{lemma:div} and the multiplicative order to describe some sets of semiprime operations.
\begin{ex}\label{ex:dvr-sp}
Let $A=B$ be a discrete valuation ring with uniformizer $\pi$. Let $\mathcal{G}:=\insid(V)^\bullet$ be the set of nonzero ideals of $A$. Both $\mathcal{G}$ and $\insid(V)$ are upward closed. By Lemma \ref{lemma:div}, we have
\begin{equation*}
(\pi^mA)^{\princ_{\pi^nA}}=(\pi^nA:(\pi^nA:\pi^mA))=\begin{cases}
(\pi^nA:A)=\pi^nA & \text{if~}m\geq n,\\
(\pi^nA:\pi^{n-m}A)=\pi^mA & \text{if~}m\leq n.
\end{cases}
\end{equation*}
In particular, $\pi^mA$ is $\princ_{\pi^nA}$-closed if and only if $m\leq n$, and thus every equivalence class with respect to $\ordine$ is a singleton (i.e., we can consider $[\mathcal{G}]=\mathcal{G}$). Moreover, $\mathcal{G}$ is a chain under $\ordine=\ordine_{(A,A,\mathcal{G})}$:
\begin{equation*}
A\ordinestretto \pi A\ordinestretto\pi^2A\ordinestretto\cdots\ordinestretto\pi^nA\ordinestretto\cdots.
\end{equation*}
The downsets of $\mathcal{G}$ are thus $\mathcal{G}$ itself (which corresponds to the identity map) and the sets $\{\pi^nA\}^\downarrow=\{A,\pi A,\ldots,\pi^nA\}$, for $n\inN$ (with $\{\pi^nA\}^\downarrow$ corresponding to the principal operation $\princ_{\pi^nA}$).

Considering also $(0)$, we see that in $(\insid(V),\ordine)$ the zero ideal is comparable only with $A$ (for which $A\ordinestretto(0)$). If each $\pi^nA$ is closed, then also $(0)$ is closed; hence, the identity map extends only to the identity map on $\insid(V)$. On the other hand, $\princ_{\pi^nA}$ extends to two different multiplicative operations, namely $\princ_{\pi^nA}\wedge\princ_{(0)}=\inf\{\princ_{\pi^nA},\princ_{(0)}\}$ and $\princ_{\pi^nA}$, according to whether $(0)$ is closed or not. In this way, we obtain exactly the semiprime operations described in \cite{vassilev_structure_2009}.

Note that $(\mathcal{G},\ordine)$ has no maximal elements, while $(\insid(V),\ordine)$ has a unique maximal element (the zero ideal) that is not a maximum. In particular, $(A,A,\mathcal{G})$ and $(A,A,\insid(V))$ do not have a canonical ideal.

\end{ex}

\begin{ex}\label{ex:dimV1}
Suppose that $V$ is a non-Noetherian one-dimensional valuation domain with valuation $\val$ and value group $\Gamma\subseteq\insR$; let $\mathcal{G}:=\insid(V)^\bullet$ and let $\Delta:=\insmult(V,V,\mathcal{G})$.

There are two classes of nonzero ideals:
\begin{itemize}
\item $P(\delta):=\{x\in V\mid \val(x)\geq\delta\}$, for $\delta\in\insR^{\geq 0}$;
\item $J(\delta):=\{x\in V\mid \val(x)>\delta\}$, for $\delta\in\insR^{\geq 0}$.
\end{itemize}
In particular, $V=P(0)$, while $J(0)$ is the maximal ideal of $V$. If $\delta\in\insR^{\geq 0}\setminus\Gamma^{\geq 0}$, then $J(\delta)=P(\delta)$.

Let $\mathcal{P}$ be the set of all $P(\delta)$, and let $\mathcal{J}$ be the set of the $J(\delta)$ with $\delta\in\Gamma^{\geq 0}$. Then, $(\mathcal{P},\mathcal{J})$ is a partition of $\mathcal{G}$. To study $\ordine=\ordine_{(A,A,\mathcal{G})}$, we start with studying it on $\mathcal{P}$ and $\mathcal{J}$.

By direct calculation, we have
\begin{equation*}
(P(\alpha):_VP(\beta))=(P(\alpha):_VJ(\beta))=\begin{cases}
V & \text{if~}\alpha<\beta\\
P(\alpha-\beta) & \text{if~}\alpha\geq\beta;
\end{cases}
\end{equation*}
doing it again, we see that $P(\alpha)\ordine P(\beta)$ if and only if $\alpha\leq\beta$, and thus $(\mathcal{P},\ordine)$ is order-isomorphic to $\insR^{\geq 0}$. Likewise,
\begin{equation*}
(J(\alpha):_VP(\beta))=\begin{cases}
V & \text{if~}\alpha<\beta\\
J(\alpha-\beta) & \text{if~}\alpha\geq\beta
\end{cases}
\end{equation*}
and
\begin{equation*}
(J(\alpha):_VJ(\beta))=\begin{cases}
V & \text{if~}\alpha<\beta\\
P(\alpha-\beta) & \text{if~}\alpha\geq\beta
\end{cases}
\end{equation*}
so $J(\alpha)\ordine J(\beta)$ if and only if $\alpha\leq\beta$, and $(\mathcal{J},\ordine)$ is  order-isomorphic to $\Gamma^{\geq 0}$. Furthermore, we see that $P(\alpha)\ordine J(\beta)$ if and only if $\alpha\leq\beta$, and $J(\alpha)\ordine P(\beta)$ if and only if $\alpha\leq\beta$. In particular, as in the discrete case, we can consider $[\mathcal{G}]=\mathcal{G}$.

Let now $\mathcal{D}:=\mathcal{G}^\star$ for some $\star\in\insmult(A,A,\mathcal{G})$: then, $\mathcal{D}$ is a downset of $\mathcal{G}$. Thus, $\mathcal{D}_1:=\mathcal{D}\cap\mathcal{P}$ is a downset of $\mathcal{P}$ (and is nonempty since $V\in\mathcal{D}_1$), while $\mathcal{D}_2:=\mathcal{D}\cap\mathcal{J}$ is either a downset of $\mathcal{J}$ or the empty set. Hence, they are intervals in the form $[0,t)$ or $[0,t]$. Set $\rho:=\sup\{\beta\mid P(\beta)\in\mathcal{D}_1\}$ and
\begin{equation*}
\gamma:=\begin{cases}
\sup\{\alpha\mid J(\alpha)\in\mathcal{D}_2\} & \text{if~}\mathcal{D}_2\neq\emptyset\\
-\infty & \text{if~}\mathcal{D}_2=\emptyset.
\end{cases}
\end{equation*}
We analyze separately $\rho$ and $\gamma$.

If $\rho<\infty$, then $\bigcap_{\alpha\in[0,\rho)}P(\alpha)=P(\rho)$, and thus $P(\rho)$ is $\star$-closed, i.e., $P(\rho)\in\mathcal{D}$. Thus, $\star\leq\princ_{P(\rho)}$. On the other hand, if $\rho=\infty$ then $\star\leq v_P$, where $v_P:=\inf\{\princ_{P(\beta)}\mid \beta\inR^{\geq 0}\}$ is the multiplicative operation closing exactly all the $P(\alpha)$. (More precisely, $v_P$ is the divisorial closure on $V$.)

If $\gamma=-\infty$ we have nothing to say. Suppose $\gamma\neq-\infty$, and let $d(\gamma):=\inf\{\princ_{J(\beta)}\mid \beta<\gamma\}$. Then, $\star\leq \princ_{P(\rho)}\wedge d(\gamma)$; furthermore, once we know $\rho$ and $\gamma$, the only ideal of which we can't say if it is closed or not is $J(\gamma)$.

If $J(\gamma)\notin\mathcal{D}_2$, then $\star=\princ_{P(\rho)}\wedge d(\gamma)$. If $J(\gamma)\in\mathcal{D}_2$, then $\star\leq \princ_{P(\rho)}\wedge d(\gamma)\wedge \princ_{J(\gamma)}=\princ_{P(\rho)}\wedge\princ_{J(\gamma)}$; in this case, moreover, $\gamma\notin\Gamma^{\geq 0}$ since otherwise $\princ_{P(\rho)}\wedge\princ_{J(\gamma)}=\princ_{P(\rho)}\wedge\princ_{P(\gamma)}=\princ_{P(\sup\{\rho,\gamma\})}$ would not close any element of $\mathcal{J}$, a contradiction.

Therefore, we have the following four possibilities:
\begin{itemize}
\item if $\rho<\infty$ and $\gamma=-\infty$ then $\star=\princ_{P(\rho)}$;
\item if $\rho=\infty$ and $\gamma=-\infty$ then $\star=v_P$;
\item if $\rho<\infty$ and $J(\gamma)\in\mathcal{D}_2$ then $\star=\princ_{P(\rho)}\wedge\princ_{J(\gamma)}$;
\item if $\rho=\infty$ and $J(\gamma)\notin\mathcal{D}_2$ (or $\gamma=\infty$) then $\star=\princ_{P(\rho)}\wedge d(\gamma)$.
\end{itemize}

If we now consider the set $\Delta':=\insmult(V,V,\insid(V))$ of all the semiprime operations, we see that $\princ_{(0)}$ closes only $(0)$ and $V$, and that no other $\princ_{J(\alpha)}$ or $\princ_{P(\beta)}$ close $(0)$. Hence, if $\rho=\infty$ or $\gamma=\infty$ then $\star$ extends uniquely to $\insid(V)$, while if $\rho<\infty$ and $\gamma<\infty$ we have exactly two extensions, namely $\star$ and $\star\wedge\princ_{(0)}$.

As in the Noetherian case, $(\mathcal{G},\ordine)$ has no maximal elements, $(\insid(V),\ordine)$ has a unique maximal element (the zero ideal) that is not a maximum, and neither $(A,A,\mathcal{G})$ nor $(A,A,\insid(V))$ have a canonical ideal.
\end{ex}

\section{Functoriality}\label{sect:functoriality}
In this section, we study how multiplicative operations behave under ring homomorphisms. We need the following definitions.
\begin{defin}
Let $A\subseteq B$ and $A'\subseteq B'$ be two ring extensions, and let $\phi:B\longrightarrow B'$ be a ring homomorphism. We say that $\phi$ is:
\begin{itemize}
\item a \emph{homomorphism of extensions} if $\phi(A)\subseteq A'$;
\item an \emph{isomorphism of extensions} if $\phi$ is an isomorphism and $\phi(A)=A'$;
\item a \emph{quotient of extensions} if $\phi$ is surjective and $\phi^{-1}(A')=A$.
\end{itemize}
\end{defin}

\begin{oss}
~\begin{enumerate}[(1)]
\item If $A\subseteq B\subseteq C$ are ring extensions, then the inclusion $\phi:B\longrightarrow C$ is a homomorphism between the extensions $A\subseteq B$ and $A\subseteq C$. Furthermore, the identity $\phi:C\longrightarrow C$ is a homomorphism between the extensions $A\subseteq C$ and $B\subseteq C$.
\item The fact that $\phi:B\longrightarrow B'$ is a homomorphisms of extensions can also be expressed by saying that the diagram
\begin{equation*}
\begin{tikzcd}
A\arrow{d}{\phi}\arrow[hook]{r} & B\arrow{d}{\phi}\\
A'\arrow[hook]{r} & B'
\end{tikzcd}
\end{equation*}
is commutative. Under this point of view, $\phi$ is a quotient if and only if $\phi$ is surjective and $A$ is the pullback of $A'$ under $\phi$.
\end{enumerate}
\end{oss}

The first step is going from multiplicative operations on $(A',B',\mathcal{G}')$ to operations on $(A,B,\mathcal{G})$.
\begin{prop}\label{prop:pullback}
Let $A\subseteq B$ and $A'\subseteq B'$ be ring extensions and let $\phi:B\longrightarrow B'$ be a homomorphism of extensions. Let $\mathcal{G}'\subseteq\inssubmod_{A'}(B')$ be upward closed in $\inssubmod_{A'}(B')$, and let $\mathcal{G}:=\phi^{-1}(\mathcal{G}')$. Then, the following hold.
\begin{enumerate}[(a)]
\item\label{prop:pullback:def} For every $\star\in\insmult(A',B',\mathcal{G}')$, the map
\begin{equation*}
\begin{aligned}
\star_\phi\colon\mathcal{G} & \longrightarrow\mathcal{G}\\
I & \longmapsto \phi^{-1}((\phi(I)A')^\star)
\end{aligned}
\end{equation*}
is a multiplicative operation on $(A,B,\mathcal{G})$.
\item\label{prop:pullback:kernel} If $L\subseteq\ker\phi$, then $L^{\star_\phi}=(\ker\phi)^{\star_\phi}=\phi^{-1}((0)^\star)$.
\item\label{prop:pullback:map} The map
\begin{equation*}
\begin{aligned}
\Psi\colon\insmult(A',B',\mathcal{G}') & \longrightarrow\insmult(A,B,\mathcal{G})\\
\star & \longmapsto \star_\phi
\end{aligned}
\end{equation*}
is well-defined and order-preserving.
\end{enumerate}
\end{prop}
\begin{proof}
We first note that $\star_\phi$ is well-defined since $(\phi(I)A')^\star\in\mathcal{G}'$ and thus its inverse image is in $\mathcal{G}$. It is also clear that $\star_\phi$ is extensive and order-preserving. To show that it is idempotent, let $I\in\mathcal{G}$. Then,
\begin{equation*}
(I^{\star_\phi})^{\star_\phi}=\phi^{-1}\left((\phi(I^{\star_\phi})A')^\star\right);
\end{equation*}
however,
\begin{equation*}
\phi(I^{\star_\phi})A'=\phi\left(\phi^{-1}(\phi(I)A')^\star\right)A'=((\phi(I)A')^\star)A'=(\phi(I)A')^\star
\end{equation*}
and thus
\begin{equation*}
(I^{\star_\phi})^{\star_\phi}=\phi^{-1}\left((\phi(I)A')^\star\right)=I^{\star_\phi}
\end{equation*}
and $\star_\phi$ is idempotent. Likewise, let $I\in\mathcal{G}$ and $b\in B$ be such that $(I:b)\in\mathcal{G}$. Then, $\phi(I:b)\in\mathcal{G}'$, and since $\mathcal{G}'$ is upward closed and $\phi(I:b)\subseteq(\phi(I):\phi(b))$ we have $(\phi(I):\phi(b))\in\mathcal{G}'$. Hence,
\begin{align*}
(I:b)^{\star_\phi} & =\phi^{-1}\left((\phi(I:b)A')^\star\right)\subseteq \\
& \subseteq\phi^{-1}\left(((\phi(I):\phi(b))A')^\star\right)\subseteq\\
& \subseteq\phi^{-1}\left((\phi(I)A':\phi(b))^\star\right)\subseteq\\
& \subseteq\phi^{-1}\left(((\phi(I)A')^\star:\phi(b))\right).
\end{align*}
Let $t\in(I:b)^{\star_\phi}$. Then, $\phi(t)\in((\phi(I)A')^\star:\phi(b))$, i.e., $\phi(t)\phi(b)=\phi(tb)\in(\phi(I)A')^\star$. Therefore, 
\begin{equation*}
tb\in\phi^{-1}\left((\phi(I)A')^\star\right)=I^{\star_\phi},
\end{equation*}
and so $t\in(I:b)^{\star_\phi}$. Thus, $(I:b)^{\star_\phi}\subseteq(I^{\star_\phi}:b)$ and $\star$ is multiplicative, as claimed.

The other two points follow directly from the definitions.
\end{proof}

\begin{prop}\label{prop:pushforward}
Let $A\subseteq B$ and $A'\subseteq B'$ be two ring extensions and let $\phi:B\longrightarrow B'$ be a quotient of extensions. Let $\mathcal{G}\subseteq\inssubmod_A(B)$ be upward closed in $\inssubmod_A(B)$, and let $\mathcal{G}':=\phi(\mathcal{G})$. Let also $\Psi$  be the map defined in Proposition \ref{prop:pullback}.  Then, the following hold.
\begin{enumerate}[(a)]
\item\label{prop:pushforward:def} For every $\star\in\insmult(A,B,\mathcal{G})$, the map
\begin{equation*}
\begin{aligned}
\star^\phi\colon\mathcal{G}' & \longrightarrow\mathcal{G}'\\
I & \longmapsto \phi(\phi^{-1}(I)^\star)
\end{aligned}
\end{equation*}
is a multiplicative operation on $(A',B',\mathcal{G}')$.
\item\label{prop:pushforward:map} The map
\begin{equation*}
\begin{aligned}
\Phi\colon\insmult(A,B,\mathcal{G}) & \longrightarrow\insmult(A',B',\mathcal{G}')\\
\star & \longmapsto \star^\phi
\end{aligned}
\end{equation*}
is well-defined and order-preserving.
\item\label{prop:pushforward:id'} $\Phi\circ\Psi$ is the identity on $\insmult(A',B',\mathcal{G}')$; that is, $(\star_\phi)^\phi=\star$ for every $\star\in\insmult(A',B',\mathcal{G}')$.
\item\label{prop:pushforward:id} If $\ker\phi\subseteq L$ for every $L\in\mathcal{G}$, then $\Psi\circ\Phi$ is the identity on $\insmult(A,B,\mathcal{G})$; that is, $(\sharp^\phi)_\phi=\sharp$ for every $\sharp\in\insmult(A,B,\mathcal{G})$.
\end{enumerate}
\end{prop}
\begin{proof}
Since $\phi^{-1}(A')=A$, each $\phi(I)$ is an $A'$-module, and thus $\star^\phi$ is well-defined. It is also clearly extensive and order-preserving. To show idempotence, we calculate
\begin{equation*}
(I^{\star^\phi})^{\star^\phi}=\phi(\phi^{-1}(I)^{\star_\phi})=\phi(\phi^{-1}(\phi(\phi^{-1}(I)^\star)))=\phi(\phi^{-1}(I)^\star)=I^{\star^\phi}.
\end{equation*}
To show that $\star^\phi$ is multiplicative, let $I\in\mathcal{G}'$ and let $b\in B'$ such that $(I:b)\in\mathcal{G}'$. Let $b=\phi(a)$ for some $a\in B$. Since $\phi$ is surjective and $\mathcal{G}$ is upward closed, $\phi^{-1}(I:b)\subseteq(\phi^{-1}(I):a)$ are both in $\mathcal{G}$. Hence, we have
\begin{equation*}
(I:b)^{\star^\phi}=\phi(\phi^{-1}(I:b)^\star)\subseteq\phi((\phi^{-1}(I):a)^\star)\subseteq\phi(\phi^{-1}(I)^\star:a).
\end{equation*}
If $t=\phi(s)\in(I:b)^{\star^\phi}$, then $sa\in\phi^{-1}(I)^\star$, and thus $tb=\phi(sa)\in I^{\star^\phi}$. Hence, $t\in(I^{\star^\phi}:b)$ and so $(I:b)^{\star^\phi}\subseteq(I^{\star^\phi}:b)$.

The second point follows directly from the definitions; for \ref{prop:pushforward:id'}, let $I\in\mathcal{G}'$ and let $L:=\phi^{-1}(I)$. Then, for every $\star\in\insmult(A',B',\mathcal{G}')$,
\begin{equation*}
I^{(\star_\phi)^\phi}=\phi(L^{\star_\phi})=\phi(\phi^{-1}(\phi(L)^\star))=\phi(L)^\star=I^\star.
\end{equation*}
To show \ref{prop:pushforward:id}, let $J\in\mathcal{G}$ and $L:=\phi(J)$. Take $\sharp\in\insmult(A,B,\mathcal{G})$. Then,
\begin{equation*}
J^{(\sharp^\phi)_\phi}=\phi^{-1}(L^{\sharp^\phi})=\phi^{-1}(\phi(\phi^{-1}(L)^\sharp))=\phi^{-1}(L)^\sharp=J^\sharp,
\end{equation*}
with the last equality coming from the fact that $\phi^{-1}(\phi(J))=J$ as $J$ contains the kernel of $\phi$. The claim is proved.
\end{proof}

\begin{cor}\label{cor:quoz}
Let $\phi:B\longrightarrow B'$ be a quotient between the extensions $A\subseteq B$ and $A'\subseteq B'$. Let $\mathcal{G}_0:=\{L\in\inssubmod_A(B)\mid \ker\phi\subseteq L\}$. For every $\mathcal{G}\subseteq\mathcal{G}_0$, we have $\insmult(A,B,\mathcal{G})\simeq\insmult(A',B',\phi(\mathcal{G}))$.
\end{cor}

\section{Principal ideal domains}\label{sect:PID}
Let $A\subseteq B$ be a ring extension. We set
\begin{equation*}
\mathcal{F}_0(A,B):=\{I\in\inssubmod_A(B)\mid IB=B\};
\end{equation*}
note that if $C$ is a further extension of $B$ then this set is also equal to $\{I\in\inssubmod_A(C)\mid IB=B\}$. To simplify the notation, we set
\begin{equation*}
\insmult_0(A,B):=\insmult(A,B,\mathcal{F}_0(A,B)).
\end{equation*}

If $D$ is an integral domain and $K$ its quotient field, $\insfracid_0(D,K)$ is just the set of nonzero $D$-submodules of $K$, and thus $\insmult_0(D,K)$ is the set of semistar operations on $D$.

If $D$ is an integral domain with quotient field $K$, an \emph{overring} of $D$ is a ring contained between $D$ and $K$; if $T$ is an overring of $D$, we set
\begin{equation*}
\mathcal{F}_p(D,T):=\{I\in\inssubmod_D(K)\mid IT=aT\text{~for some~}a\in K\}.
\end{equation*}

\begin{teor}\label{teor:principali}
Let $D$ be an integral domain with quotient field $K$, let $T$ be an overring of $D$, and let $\Delta(T):=\{\star\in\insmult(D,K,\mathcal{F}_p(D,T))\mid T=T^\star\}$. Then, the restriction map
\begin{equation*}
\begin{aligned}
\rho\colon\Delta(T) & \longrightarrow\insmult_0(D,T)\\
\star & \longmapsto\star|_{\mathcal{F}_0(D,T)}
\end{aligned}
\end{equation*}
is an order isomorphism.
\end{teor}
\begin{proof}
To simplify the notation, let $\mathcal{G}:=\mathcal{F}_p(D,T)$ and $\mathcal{G}_0:=\mathcal{F}_0(D,T)$.

Let $\star\in\Delta(T)$. If $I\in\mathcal{G}_0$, then $I\subseteq T$; since $T=T^\star$, we have $I\subseteq I^\star\subseteq T$ and thus $T=IT\subseteq I^\star T\subseteq T$, i.e., $I^\star T=T$ or $I^\star\in\mathcal{G}_0$. Hence, $\rho(\star)$ is well-defined.

Clearly, $\rho(\star)$ is a closure operation; furthermore, it is multiplicative as it is the restriction of a multiplicative operation. Thus, $\rho$ is well-defined. Since $I^\star=I^{\rho(\star)}$ for every $I\in\mathcal{G}_0$, $\rho$ is injective.

We now show that $\rho$ is surjective. Let $\star\in\Delta(T)$, and define
\begin{equation*}
\begin{aligned}
\sharp\colon\mathcal{G} & \longrightarrow\mathcal{G},\\
I & \longmapsto a(a^{-1}I)^\star
\end{aligned}
\end{equation*}
where $a$ is any element of $K$ such that $IT=aT$. Then, $\sharp$ is well-defined since if $IT=a'T$ then $a'=ua$ for some unit $u$ of $T$, and
\begin{equation*}
a(a^{-1}I)^\star=auu^{-1}(a^{-1}I)^\star=au(u^{-1}a^{-1}I)^\star=a'(a'^{-1}I)^\star.
\end{equation*}
In particular, $\sharp$ is an extension of $\star$.

The map $\sharp$ is clearly extensive, and $(bI)^\sharp=bI^\sharp$ for every $b\in K$ and every $I\in\mathcal{G}_0$. Since $J\subseteq J^\star\subseteq T$ for every $J\in\insfracid_0(D,T)$, we have $I^\sharp T=IT$ for every fractional ideal $I$, and thus
\begin{equation*}
(I^\sharp)^\sharp=a(a^{-1}I^\sharp)^\star=a(a^{-1}a(a^{-1}I)^\star)^\star=a(a^{-1}I)^\star=I^\sharp,
\end{equation*}
that is, $\sharp$ is idempotent.

To show that it is order-preserving, let $I\subseteq J$ be two elements of $\mathcal{G}$, and suppose $IT=aT$, $JT=bT$. By multiplying for $b^{-1}$, we can suppose without loss of generality that $JT=T$; in particular, $a\in T$. We have $a^{-1}I\in\mathcal{G}_0$, and thus
\begin{equation*}
a^{-1}I\subseteq a^{-1}J\cap T\subseteq a^{-1}J^\star\cap T=(J^\star:_Ta).
\end{equation*}
The ideal $(J^\star:_Ta)$ is contained between $a^{-1}I$ and $T$, and thus it is in $\mathcal{G}_0$. Since $\star$ is multiplicative, it follows that $(J^\star:_Ta)$ is also $\star$-closed; hence, $(I:_Ta)^\star\subseteq(J^\star:_Ta)\subseteq(J^\star:_Ta)$, that is, $a(a^{-1}I)^\star\subseteq J^\star$. Since $JT=T$, we have $J^\star=J^\sharp$, and so $I^\sharp\subseteq J^\sharp$.

Hence, every $\star\in\insmult(D,T,\mathcal{G}_0)=\insmult_0(D,T)$ can be extended to a multiplicative operation on $(D,K,\mathcal{G})$, and thus $\rho$ is surjective. Since $\rho$ is clearly order-preserving, it follows that $\rho$ is an order isomorphism.
\end{proof}

The most obvious application of Theorem \ref{teor:principali} is when $T$ is a principal ideal domain; with two additional lemmas, we can say something more.
\begin{lemma}\label{lemma:explodes}
Let $D$ be an integral domain, and let $T$ be an overring. If $J$ is a $D$-module such that $JT=T$, then $(D:T)\subseteq J$.
\end{lemma}
\begin{proof}
By hypothesis, $1\in JT$, and thus we can find $j_1,\ldots,j_n\in J$, $t_1,\ldots,t_n\in T$ such that $1=j_1t_1+\cdots+j_nt_n$. If $z\in(D:T)$, then
\begin{equation*}
z=(j_1t_1+\cdots+j_nt_n)z=j_1t_1z+\cdots+j_nt_nz\in j_1D+\cdots+j_nD\subseteq J
\end{equation*}
since $t_iz\in T(D:T)\subseteq D$ for every $i$. Hence, $(D:T)\subseteq J$.
\end{proof}

\begin{lemma}\label{lemma:PIDFover}
Let $D$ be an integral domain. There is at most one overring of $D$ that is both a principal ideal domain and a fractional ideal of $D$; furthermore, if it exists then it is the biggest overring of $D$ that is also a fractional ideal.
\end{lemma}
\begin{proof}
Suppose there are two, say $T_1$ and $T_2$; without loss of generality, $T_2\nsubseteq T_1$. Then, $T_1T_2$ is a proper overring of $T_1$ which is also a fractional ideal over $T_1$; however, since $T_1$ is a principal ideal domain this is impossible. The last claim follows in the same way.
\end{proof}

\begin{teor}\label{teor:PID}
Let $D$ be an integral domain, and let $T$ be an overring of $D$ that is a principal ideal domain. If $I:=(D:T)\neq(0)$, then the sets $\insfstar(D)$, $\insmult_0(D,T)$ and $\insmult_0(D/I,T/I)$ are order-isomorphic.
\end{teor}
\begin{proof}
Since $T$ is a PID, $\mathcal{F}_p(D,T)$ coincides with the set of fractional ideals of $D$. Since $(D:T)\neq (0)$, $T$ is a fractional ideal of $D$; by Lemma \ref{lemma:PIDFover}, $T$ is the largest overring of $D$ that is a fractional ideal of $D$. Hence, $T=T^\star$ for every fractional star operation $\star$ and so (in the terminology of Theorem \ref{teor:principali}) $\Delta(T)=\insmult(D,K,\insfracid(D)^\nz)$. As $\insfstar(D)=\insmult(D,K,\insfracid(D)^\nz)$, by Theorem \ref{teor:principali} we have an isomorphism between $\insfstar(D)$ and $\insmult_0(D,T)$.

The map $\phi:T\longrightarrow T/I$ is a quotient between the extensions $D\subseteq T$ and $D/I\subseteq T/I$. Each member of $\insfracid_0(D,T)$ contains $\ker\phi=I$ (Lemma \ref{lemma:explodes}); moreover, $\phi(\insfracid_0(D,T))=\insfracid_0(D/I,T/I)$. Thus, by Corollary \ref{cor:quoz}, we have $\insmult_0(D,T)\simeq\insmult_0(D/I,T/I)$. The claim is proved.
\end{proof}

\begin{cor}
Let $D$ be an integral domain, and let $T$ be an overring of $D$ that is a principal ideal domain. If $I:=(D:T)\neq(0)$, then there is an isomorphism between $\insstar(D)$ and $\{\star\in\insmult_0(D/I,T/I)\mid D/I=(D/I)^\star\}$.
\end{cor}
\begin{proof}
The isomorphisms of Theorem \ref{teor:PID} preserve whether $D$ is closed.
\end{proof}

\begin{cor}\label{cor:isoext}
Let $D,D'$ be integral domains, and let $T,T'$ be principal ideal domains that are overrings of $D$ and $D'$, respectively. If the ring extensions $D/(D:T)\subseteq T/(D:T)$ and $D'/(D':T')\subseteq T'/(D':T')$ are isomorphic, then $\insstar(D)\simeq\insstar(D')$.
\end{cor}

\begin{oss}\label{oss:twoext}
In the proof of Theorem \ref{teor:principali}, the restriction map $\rho:\insfstar(D)\longrightarrow\insmult_0(D,T)$ can also be seen as the composition of two restrictions: the first one, say $\rho_1$, going from $\insfstar(D)=\insmult(D,K,\insfracid(D)^\nz)$ to $\insmult(D,T,\inssubmod_D(T)^\nz)$, and the second one going from $\insmult(D,T,\inssubmod_D(T)^\nz)$ to $\insmult(D,T,\insfracid_0(D,T))=\insmult_0(D,T)$. However, while $\rho_1$ is injective, it is not surjective: for example, the map $\star$ sending every $I\in\inssubmod_D(T)$ to $T$ is a multiplicative operation on $(D,T,\inssubmod_D(T)^\nz)$, as $(I^\star:b)=(T:b)=T$ for every $b\in T$, but it clearly cannot be the restriction of a fractional star operation.

As a less trivial example, let $T:=K[[X]]$ and $D:=K[[X^3,X^4,X^5]]=K+X^3K[[X]]$; then, $(D:T)=\mathfrak{m}_D=X^3K[[X]]$. Let $I:=D+X^2D=K+X^2K[[X]]$, and let $J:=XD+\mathfrak{m}_D=XK+X^3K[[X]]$. Then, $J=XI$, and thus if $\star$ is a star operation on $D$ then $I$ is $\star$-closed if and only if $J$ is $\star$-closed.

Let $\mathcal{G}$ be the set of $D$-submodules of $T$ contained between $\mathfrak{m}_D$ and $T$. Consider the multiplicative operation $\sharp$ on $(D,T,\mathcal{G})$ generated by $D$ and $I$: then, $J^\sharp=(D:_T(D:_TJ))\cap(I:_T(I:_TJ))$. We have
\begin{equation*}
(D:_TJ)=X^2K[[X]] \Longrightarrow (D:_T(D:_TJ))=XK[[X]]
\end{equation*}
and, likewise,
\begin{equation*}
(I:_TJ)=XK[[X]] \Longrightarrow (I:_T(I:_TJ))=XK[[X]]
\end{equation*}
so that $J^\sharp=XK[[X]]\neq J$. In the Artinian setting, this means that the restriction map from $\insmult(K,K[[X]]/(X^3),\inssubmod_K(K[[X]]/(X^3)))$ to $\insmult_0(K,K[[X]]/(X^3))$ is not an isomorphism.
\end{oss}

The way to obtain Theorem \ref{teor:PID} from Theorem \ref{teor:principali} can also be used in a slightly more general setting.
\begin{prop}\label{prop:princ-gen}
Let $D$ be an integral domain, and let $T$ be an overring of $D$ that is a fractional ideal of $D$. Let $I:=(D:T)$.
\begin{enumerate}[(a)]
\item\label{prop:princ-gen:Tdiv} If every $I\notin\insfracid_p(D,T)$ is divisorial over $T$, then there are isomorphisms between $\{\star\in\insfstar(D)\mid T=T^\star\}$, $\insmult_0(D,T)$ and $\insmult_0(D/I,T/I)$.
\item\label{prop:princ-gen:Ddiv} If every $I\notin\insfracid_p(D,T)$ is divisorial over $D$, then there are isomorphisms between $\insstar(D)$, $\{\star\in\insmult_0(D,T)\mid D=D^\star\}$ and $\{\star\in\insmult_0(D/I,T/I)\mid (D/I)=(D/I)^\star\}$
\end{enumerate}
\end{prop}
\begin{proof}
Let $\mathcal{N}:=\insfracid(D)\setminus\insfracid_p(D,T)$. By hypothesis, all elements of $\mathcal{N}$ are nondivisorial over $T$ (case \ref{prop:princ-gen:Tdiv}) or over $D$ (case \ref{prop:princ-gen:Ddiv}).

\ref{prop:princ-gen:Tdiv} As in the proof of Theorem \ref{teor:principali}, we are going to show that the restriction map $\rho$ from $\Delta(T):=\{\star\in\insfstar(D)\mid T=T^\star\}$ to $\insmult_0(D,T)$ is an isomorphism; the fact that $\rho$ is well-defined and injective follows in the same way.

To show that it is surjective, let $\star\in\insmult_0(D,T)$, and define
\begin{equation*}
\begin{aligned}
\sharp\colon\mathcal{F}(D) & \longrightarrow\mathcal{F}(D),\\
I & \longmapsto \begin{cases}
a(a^{-1}I)^\star & \text{if~}IT=aT,\\
I & \text{if~}I\in\mathcal{N}.
\end{cases}
\end{aligned}
\end{equation*}
The restriction of $\sharp$ to $\insfracid_p(D,T)$ is a multiplicative operation on $(D,T,\insfracid_p(D,T))$; it is clear that $\sharp$ itself is extensive and idempotent.

To show that it is order-preserving, let $I\subseteq J$. If $I\in\mathcal{N}$ or if both $I,J\in\insfracid_p(D,T)$ then clearly $I^\sharp\subseteq J^\sharp$. Suppose $I\in\insfracid_p(D,T)$ and $J\in\mathcal{N}$; without loss of generality, $IT=T$. Then,
\begin{equation*}
I^\sharp=I^\star\subseteq(T:_T(T:_TI))=(T:_T(T:_TIT))=T=IT.
\end{equation*}
However, as $J\in\mathcal{N}$, by hypothesis $J$ is divisorial over $T$, and thus in particular it is a $T$-ideal; hence, $I^\sharp=IT\subseteq JT=T=J^\sharp$, and so $\sharp$ is order-preserving. The fact that $\sharp$ is multiplicative follows immediately by the multiplicativity of $\sharp$ on $\insfracid_p(D,T)$ and by the fact that $(I:b)=b^{-1}I\in\mathcal{N}$ if and only if $I\in\mathcal{N}$.

Therefore, $\rho$ is surjective and thus an isomorphism, as claimed. The claim about the quotient follows as in the proof of Theorem \ref{teor:PID}.

\ref{prop:princ-gen:Ddiv} is proved similarly: $\rho$ is well-defined and injective, $\sharp$ is defined in the same way and the only problem is showing that $\sharp$ is order-preserving, with the only non-trivial case being when  $I\in\insfracid_p(D,T)$ and $J\in\mathcal{N}$. Suppose we are in this case, and without loss of generality suppose $IT=T$. Then, $I^\sharp=I^\star\subseteq(D:_T(D:_TI))$; since $IT=T$, we have
\begin{equation*}
(D:_KI)\subseteq(D:_KI)T\subseteq(T:_KIT)=T,
\end{equation*}
and so $(D:_KI)=(D:_TI)$. Therefore,
\begin{equation*}
(D:_T(D:_TI))=(D:_T(D:_KI))\subseteq(D:_K(D:_KI))=I^v.
\end{equation*}
However, as $J\in\mathcal{N}$, we have $J=J^v\supseteq I^v$; thus, $I^\sharp\subseteq J=J^\sharp$, and so $\sharp$ is order-preserving. The fact that $\sharp$ is multiplicative follows immediately by the multiplicativity of $\sharp$ on $\insfracid_p(D,T)$ and by the fact that $(I:b)=b^{-1}I\in\mathcal{N}$ if and only if $I\in\mathcal{N}$. Thus, $\rho$ is an isomorphism, as claimed. The part about the quotient follows as in the proof of Theorem \ref{teor:PID}.
\end{proof}

Theorem \ref{teor:PID} and Proposition \ref{prop:princ-gen} can be seen as generalizations of some already known results. 

Recall that a \emph{conductive domain} is an integral domain $D$ such that $(D:T)\neq(0)$ for all overrings $T$ of $D$ different from the quotient field of $D$. If $D$ is conductive and seminormal, then by \cite[Proposition 2.12(i)]{conductivedomains} $D$ is the pullback of the diagram
\begin{equation*}
\begin{tikzcd}
& A\arrow{d}\\
V\arrow[two heads]{r}{\phi} & K,
\end{tikzcd}
\end{equation*}
where $V$ is a valuation overring of $D$ and $K$ is the residue field of $V$.

The ``in particular'' statement of the following proposition is part of \cite[Theorem 2.5]{hmp_finite}.
\begin{prop}\label{prop:conductive}
Let $D,V,K,A$ be as above. Then, there is an isomorphism between $\insstar(D)$ and $\{\star\in\insmult_0(A,K)\mid D=D^\star\}$. In particular, if $K$ is the quotient field of $A$ then there is an isomorphism between $\insstar(D)$ and $\{\star\in\inssemistar(A)\mid A=A^\star\}$.
\end{prop}
\begin{proof}
Let $\valut$ be the valuation relative to $V$. Take an $I\in\insfracid(D)\setminus\insfracid_p(D,V)$: we claim that $I=\bigcap\{aD\mid I\subseteq aD\}$. Since $IV$ is not principal, $\valut(I)$ has no infimum in the value group of $V$. Let $x$ be in the intersection, and suppose there is an $i\in I$ such that $\valut(i)\leq\valut(x)$. Then, there is a $j\in I$ such that $\valut(j)<\valut(x)$, and thus $\valut(xj^{-1})>0$, i.e., $xj^{-1}$ belongs to the maximal ideal of $V$, which by construction is $(D:V)$. Therefore, $x\in j(D:V)\subseteq jD\subseteq I$. Thus, $I$ is equal to the intersection and so it is $D$-divisorial. The claim now follows from Proposition \ref{prop:princ-gen}\ref{prop:princ-gen:Ddiv}.

The ``in particular'' statement follows since if $K$ is the quotient field of $A$ then $\inssemistar(A)=\insmult_0(A,K)$.
\end{proof}

Recall that a domain $D$ is a \emph{pseudo-valuation domain} (PVD) if there is a valuation overring $V$ of $D$ such that $(D:V)$ is the maximal ideal of $V$; in this case, $V$ is called the valuation overring \emph{associated} to $D$ \cite{pvd}. If $D$ is a PVD, $F$ is the residue field of $D$ and $L$ the residue field of $V$, the multiplicative operations on $(F,L,\insfracid_0(F,L))$ closing $F$ are exactly what in \cite[Section 3]{pvd-star} were called ``$F$-star operations on $L$'' (since $\insfracid_0(F,L)=\inssubmod_F(L)^\bullet$). Under this terminology, the following corollary is a different form of \cite[Theorem 3.1]{pvd-star}.
\begin{cor}
Let $D$ be a pseudo-valuation domain with associated valuation overring $V$, and let $F$ and $L$ be their respective residue field. Then, $\insstar(D)$ is isomorphic to the set of the $\star\in\insmult_0(F,L)$ such that $F=F^\star$.
\end{cor}
\begin{proof}
A pseudo-valuation domain is seminormal and conductive, and thus we can apply Proposition \ref{prop:conductive}.
\end{proof}

Another consequence is the case of Pr\"ufer domains, which was treated in \cite[Theorem 2.5]{hmp_finite}.
\begin{cor}
Let $D$ be a Pr\"ufer domain that is not a valuation domain, and let $P\subseteq\Jac(D)$ be a nonzero prime ideal. Then, there is an isomorphism between between $\insstar(D)$ and $\inssmstar(D/P)$, where $\inssmstar(D/P)$ is the set of semistar operations on $D/P$ closing $D/P$.
\end{cor}

The case of fractional star operations is slightly more delicate (since $P$ and $D_P$ may not be closed by every such closure), but can be treated similarly. See \cite[Proposition 2.2]{starloc2} for a description.

\section{Artinian extensions}\label{sect:Artin}
In this section, we study the consequences of the previous results on one-dimensional Noetherian domains. In view of \cite[Theorem 2.3]{houston_noeth-starfinite} and \cite[Theorem 5.4]{starloc}, it is not reductive to consider only local domains.

Suppose thus that $D$ is a local one-dimensional Noetherian domain, and suppose furthermore that its integral closure $T$ is finite over $D$ (i.e., $R$ is analytically unramified). Then, $T$ is a principal ideal domain and $(D:T)\neq(0)$: by Corollary \ref{cor:isoext} the sets $\insstar(D)$ and $\insfstar(D)$ depend only on the extension $D/(D:T)\subseteq T/(D:T)$, which is an extension of Artinian rings, with $A:=D/(D:T)$ local and $B:=T/(D:T)$ a principal ideal ring. In particular, $B$ can be written as a product $B_1\times\cdots\times B_t$, where each $B_i$ is an Artinian local ring that is also an $A$-algebra and a principal ideal ring. 

\begin{ex}\label{ex:DTmD}
Suppose $A=k$ is a field; this corresponds to the case where $(D:T)$ is the maximal ideal $\mathfrak{m}_D$ of $D$. Then, each $B_i$ is a $k$-algebra; furthermore, by \cite[Theorem 8]{struct-PIR}, $B_i$ is isomorphic to $L_i[[X]]/(X^{e_i})\simeq L_i[X]/(X^{e_i})$, where $L_i$ is the residue field of $B_i$ and $e_i\geq 1$ is a natural number. In particular, setting $f_i:=[L_i:k]$, we have
\begin{equation*}
\ell_D(D/(D:T))=\ell_A(B)=\sum_{i=1}^te_if_i,
\end{equation*}
and thus, if $n:=\ell_A(B)$ is fixed, there are at most a finite number of choices for $e_i$ and $f_i$. Moreover, if $k$ is finite, then $f_i$ uniquely determines $L_i$; hence, there are only finitely many possibilities for $B$, and thus finitely many possibilities for $\insstar(D)$ and $\insfstar(D)\simeq\insmult_0(k,B)$.

For example, if $n=3$, we have only five possible cases:
\begin{itemize}
\item $t=1$:
\begin{itemize}
\item $e=1$, $f=3$;
\item $e=3$, $f=1$;
\end{itemize}
\item $t=2$:
\begin{itemize}
\item $e_1=1$, $f_1=2$; $e_2=f_2=1$;
\item $e_1=2$, $f_1=1$; $e_2=f_2=1$;
\end{itemize}
\item $t=3$:
\begin{itemize}
\item $e_i=f_i=1$ for $i\in\{1,2,3\}$.
\end{itemize}
\end{itemize}
Those cases correspond exactly to the five cases of \cite[Theorem 3.1]{houston_noeth-starfinite}: more precisely, they correspond, respectively, to points (3), (4), (5), (6) and (2) of the theorem. In particular, our method gives an ``high-level'' justification for the fact that the study of star operations splits into these cases.
\end{ex}

The case $(D:T)=\mathfrak{m}_D$ will be studied in more detail in \cite{asymptotics-star}.

Example \ref{ex:DTmD} can be generalized. If $n,q<\infty$, we denote by $\mathcal{C}(n,q)$ the set of domains $D$ such that:
\begin{itemize}
\item $D$ is one-dimensional, local and Noetherian;
\item the residue field of $D$ has cardinality $q$;
\item the integral closure $T$ of $D$ is finite over $D$ and $\ell_D(T/(D:T))=n$.
\end{itemize}

\begin{prop}\label{prop:finposs}
Fix $n,q<\infty$. Then, there are only finitely many possibilities for $\insstar(D)$ and $\insfstar(D)$, as $D$ ranges in $\mathcal{C}(n,q)$.
\end{prop}
\begin{proof}
Take any $D\in\mathcal{C}(n,q)$ and let $A:=D/(D:T)$ and $B:=T/(D:T)$. By Theorem \ref{teor:PID} and Corollary \ref{cor:isoext}, the sets $\insstar(D)$ and $\insfstar(D)$ depend only on $A\subseteq B$; hence, it is enough to show that there are only finitely many possibilities for the extension $A\subseteq B$.

Since $\ell_A(B)=\ell_D(T/(D:T))$, we have $\ell_A(A)\leq n$: hence, $\ell_A(A)\leq n$ and thus $|A|\leq q^n$, so that there are only finitely many possible structures for $A$. Likewise, $\ell_A(B)=\ell_A(A)+\ell_A(B/A)\leq 2n$, and thus $|B|\leq q^{2n}$; hence, there are only finitely many possible extensions $A\subseteq B$. The claim is proved.
\end{proof}

We can interpret Proposition \ref{prop:finposs} by saying that, once $n$ and $q$ are fixed, we can find a finite family of integral domains that ``represents'' all one-dimensional local domains with $|D/\mathfrak{m}_D|=q$ and $\ell_D(T/(D:T))=n$, in the sense that for any such domain $D$ the sets $\insstar(D)$ and $\insfstar(D)$ are isomorphic to $\insstar(T)$ and $\insfstar(T)$ for some member $T$ of the family. We now want to show that, under some hypothesis, these members can be taken to be close to a polynomial ring.

\begin{lemma}\label{lemma:constructPID}
Let $k$ be a field and let $L_1,\ldots,L_t$ be finite algebraic extensions of $k$ (not necessarily distinct). Then, the following hold.
\begin{enumerate}[(a)]
\item\label{lemma:constructPID:gen} There are an integer $m$ and a principal ideal domain $T$ that is an overring of $k[X_1,\ldots,X_m]$ such that $T$ has $t$ maximal ideals, and the residue fields of $T$ are exactly $L_1,\ldots,L_t$.
\item\label{lemma:constructPID:simple} If each $L_i$ is a simple extension of $k$, then $T$ can be taken to be a localization of $k[X_1,\ldots,X_m]$.
\item\label{lemma:constructPID:inf} If each $L_i$ is a simple extension of $k$ and $k$ is infinite, we can take $m=1$.
\end{enumerate}
\end{lemma}
\begin{proof}
\ref{lemma:constructPID:gen} Let $N_1,\ldots,N_t$ be prime ideals of $k[X]$ with residue field $k$, and let $T_1:=S^{-1}k[X]$ with $S:=k[X]\setminus\bigcup_iN_i$. Using repeatedly the construction in \cite[Section 2]{heitmann-PID}, we can construct domains $T_1\subseteq T_2\subseteq\cdots$, such that, for $i>1$, $T_i$ is a principal ideal overring of $T_{i-1}[X_i]$ (and so of $k[X_1,\ldots,X_i]$) and such that one of the following holds: if the residue fields of $T_{i-1}$ are $F_1,\ldots,F_r$, then the residue fields of $T_i$ are either $F_1,\ldots,F_r,F_i$ (for some $1\leq i\leq r$) or $F_1,\ldots,F'_i,\ldots,F_r$, with $F'_i$ being a simple extension of $F_i$. In particular, after finitely many steps the residue fields will be $L_1,\ldots,L_n$, as claimed.

\ref{lemma:constructPID:simple} For each $i$, let $p_i(X)\in k[X]$ be a polynomial whose splitting field is $L_i$; then, the requested $T$ will be $S^{-1}k[X_1,\ldots,X_t]$, where $S$ is the multiplicatively closed set $k[X]\setminus\bigcup_i(p_i(X_i))$. 

\ref{lemma:constructPID:inf} For a given $L\in\{L_1,\ldots,L_t\}$, there are infinitely many irreducible monic polynomials in $k[X]$ whose splitting field is $L$: since $L$ is simple there will be one, say $p(X)$, and the others will be those in the form $p(X+s)$ for $s\in k$ (since $p(X+s)=p(X+s')$ can happen only finitely many times for each $s$). Hence, for every $L$ there are are infinitely many maximal ideals of $k[X]$ with residue field $L$. In particular, we can take distinct maximal ideals $M_1,\ldots,M_n$ of $k[X]$ such that $k[X]/M_i\simeq L_i$; the requested $T$ will be $S^{-1}k[X]$, where $S:=k[X]\setminus\bigcup_iM_i$.
\end{proof}

\begin{oss}
The integer $m$ obtained in the proof of Lemma \ref{lemma:constructPID} is not tight: for example, it is possible than $k[X]$ has already $L_1,\ldots,L_t$ as residue fields even if $k$ is finite (e.g., if each $L_i$ appears only once, or more generally if it does not appear too many times). If each $L_i$ is actually $k$, an application of the methods of \cite[Section 2]{heitmann-PID} gives an upper bound $m\leq 1+\log_{|k|}t$ (or $m=1$ if $k$ is infinite, as in part \ref{lemma:constructPID:inf} of the lemma).
\end{oss}

\begin{prop}\label{prop:stdform}
Let $D$ be a one-dimensional local Noetherian domain with residue field $k$ of characteristic $p$; let $T$ be the integral closure of $D$ and let $n:=|\Max(T)|$. Suppose that $(D:T)\neq(0)$ and that $p\in(D:T)$. Then, the following hold.
\begin{enumerate}[(a)]
\item\label{prop:stdform:perf} If $k$ is perfect and infinite, then there are a localization $T'$ of $k[X]$ and a one-dimensional local Noetherian domain $R\subseteq T'$, with integral closure $T'$, such that $\insstar(D)\simeq\insstar(R)$.
\item\label{prop:stdform:gen} There are an integer $m$, a principal ideal overring $T'$ of $k[X_1,\ldots,X_m]$ and a one-dimensional local Noetherian domain $R\subseteq T'$, with integral closure $T'$, such that $\insstar(D)\simeq\insstar(R)$.
\item\label{prop:stdform:pow} If $T$ is local with residue field $L$, there is an integral domain $R\subseteq L[[X]]$, with integral closure $L[[X]]$, such that $\insstar(D)\simeq\insstar(R)$.
\end{enumerate}
\end{prop}
\begin{proof}
Since $p\in(D:T)$, the quotient $D/(D:T)$ is a $k$-algebra, and thus so is $B:=T/(D:T)$; let $B=B_1\times\cdots\times B_n$. Then, also all the $B_i$ are $k$-algebras. By \cite[Theorem 8]{struct-PIR}, each $B_i$ is isomorphic to $L_i[[X]]/(X^{e_i})\simeq L_i[X]/(X^{e_i})$, where $L_i$ is the residue field of $B_i$ and the $e_i\geq 1$ are integers. Furthermore, if $B$ is local then $n=1$ and so $B\simeq L[[X]]/(X^e)$. Moreover, each $L_i$ is finite over $k$ since $T$ is finite over $D$.

By Lemma \ref{lemma:constructPID}, we can find a principal ideal domain $T'$ with residue fields $L_1,\ldots,L_n$ which is an overring of $k[X]$ (if $k$ is infinite and perfect) or of some $k[X_1,\ldots,X_m]$. Let $M_1,\ldots,M_n$ be the maximal ideals of $T'$, and let $I:=M_1^{e_1}\cdots M_n^{e_n}$: then, $T'/I\simeq B$. Consider the pullback
\begin{equation*}
\begin{tikzcd}
R\arrow[hook]{d}\arrow[two heads]{r}{\pi} & A\arrow[hook]{d}\\
T'\arrow[two heads]{r}{\pi} & B.
\end{tikzcd}
\end{equation*}
Since $A$ is local, also $R$ is local; furthermore, since $T'$ is a domain then $R$ too is a domain with the same quotient field of $T'$ \cite[Corollary 1.5(7)]{topologically-defined}. Moreover, since $B$ is finite over $A$ then $T'$ is finite over $R$ \cite[Corollary 1.5(4)]{topologically-defined}, and since $A$ and $B$ are Noetherian then $R$ is Noetherian too \cite[Proposition 1.8]{topologically-defined}. To summarize, $R$ is a one-dimensional local Noetherian domain with integral closure $T$. By Theorem \ref{teor:PID}, $\insfstar(R)\simeq\insmult_0(A,B)\simeq\insfstar(D)$, and likewise $\insstar(R)\simeq\insstar(D)$. The claim is proved.
\end{proof}

In the local case, this result allows to calculate $\insstar(D)$ through a fairly explicit domain $D$, with the advantage of working on integral domains (instead of rings with zero-divisors). For example, if the integral closure $T$ of $D$ is local and $(D:T)=\mathfrak{m}_D$, then we can calculate $\insstar(D)$ by considering instead the ring $F+X^nL[[X]]$, for some $n$ and some field extension $F\subseteq L$.

\bibliographystyle{plain}
\bibliography{/bib/miei,/bib/articoli,/bib/libri}
\end{document}